\documentclass[12pt]{article}
\usepackage{amsmath,amsfonts,amsthm,amssymb,mathtools,enumerate}
\usepackage{mathrsfs, graphicx,color,latexsym, tikz, calc
}
\usepackage[numbers]{natbib}
\usepackage{booktabs}
\usepackage{multirow}
\usepackage{longtable}
\usepackage{array}
\usepackage{amsmath}
\usepackage{tikz}
\usetikzlibrary{arrows.meta}
\allowdisplaybreaks

\usepackage[colorlinks=true,
linkcolor=blue,citecolor=blue,
urlcolor=blue]{hyperref}
\usetikzlibrary{shadows}
\usetikzlibrary{patterns,arrows,decorations.pathreplacing}
\usepackage{rotating}
\usepackage{booktabs}
\usepackage{pifont}
\usepackage{floatrow}
\usepackage{caption}
\usepackage{longtable}
\usepackage{enumitem}

\newtheorem{theorem}{Theorem}
\newtheorem{lemma}{Lemma}

\newtheorem{conjecture}{Conjecture}

\newtheorem{corollary}{Corollary}

\newtheorem{claim}{Claim}
\newcommand{\ntriangle}{%
\mathrel{\ooalign{$\triangle$\cr\hidewidth$/$\hidewidth\cr}}%
}

\newcommand{\ntri}{\mathrel{\rlap{\kern 1pt$\not$}\Delta}}

\title{\bf \Large On the sum of the two largest eigenvalues of the curl-curl operator on graphs}

\author{{Yueli Han$^a$, \ \ Lu Lu$^{a,}$\footnote{Corresponding author.}\setcounter{footnote}{-1}\footnote{E-mail address: hanyueli2022@163.com (Y. Han), lulumath@csu.edu.cn (L. Lu), jfwang@sdut.edu.cn(J. Wang).}, \ \ Jianfeng Wang$^b$}\\[2mm]
\footnotesize $^a$School of Mathematics and Statistics, HNP-LAMA, Central South University\\
\footnotesize Changsha, Hunan, 410083, China\\
\footnotesize $^b$School of Mathematics and Statistics, Shandong University of Technology\\
\footnotesize Zibo, Shandong, 255049, China}

\date{ }
\begin{document}
	\maketitle
	\begin{abstract}
The Grone--Merris conjecture, proved by Bai in~2011, states that the spectrum of the graph Laplacian $\Delta_0 = -\operatorname{div}\operatorname{grad}$ is majorized by the conjugate of the vertex degree sequence.
Duval and Reiner proposed a simplicial complex analogue of this statement. On a graph, where triangles serve as $2$-simplices, their conjecture reduces to the assertion that the spectrum of $\operatorname{curl}^*\operatorname{curl}$ is majorized by the conjugate of the second-order degree sequence, which records the number of triangles containing each vertex.

We prove that the sum of the two largest eigenvalues of $\operatorname{curl}^*\operatorname{curl}$ does not exceed the sum of the first two entries of that conjugate sequence.
This confirms the first two majorization inequalities predicted by Duval and Reiner for $\operatorname{curl}^*\operatorname{curl}$.
As a corollary, we obtain upper bounds for the two largest eigenvalues of the full graph Helmholtzian $\Delta_1 = -\operatorname{grad}\operatorname{div} + \operatorname{curl}^*\operatorname{curl}$.
The same result extends to the up-Laplacian of any $3$-family, yielding a concrete step towards the Duval--Reiner conjecture in dimension~$1$.
		\end{abstract}
	
	


\section{Introduction}

Let $G$ be a simple graph with vertex set $V(G)$ and edge set $E(G)$, and let $T(G)$ denote the set of triangles in $G$.
We consider real-valued functions
\[
f:  V(G)\to\mathbb{R},\qquad 
X: V(G)\times V(G)\to\mathbb{R},\qquad 
\Phi:  V(G)\times V(G)\times V(G)\to\mathbb{R},
\]
where $X$ and $\Phi$ are alternating on $E(G)$ and $T(G)$, respectively. More precisely,
\[
X(v_i,v_j)=
\begin{cases}
-X(v_j,v_i), & \text{if } \{v_i,v_j\}\in E(G),\\
0, & \text{otherwise},
\end{cases}
\]
and for every $\{v_i,v_j,v_k\}\in T(G)$,
\[
\Phi(v_i,v_j,v_k)=\Phi(v_j,v_k,v_i)=\Phi(v_k,v_i,v_j)
=-\Phi(v_j,v_i,v_k)=-\Phi(v_i,v_k,v_j)=-\Phi(v_k,v_j,v_i),
\]
while $\Phi(v_i,v_j,v_k)=0$ if $\{v_i,v_j,v_k\}\notin T(G)$.
We denote by $L^2(V(G))$, $L^2_{\wedge}(E(G))$ and $L^2_{\wedge}(T(G))$ the spaces of functions of the above three types; the subscript $\wedge$ emphasizes the alternating property.

In the language of algebraic topology, $f$, $X$ and $\Phi$ are $0$-, $1$- and $2$-cochains, respectively. They can be viewed as discrete analogues of differential forms on a smooth manifold \cite{FWW1983} and are often called discrete differential forms \cite{DLM2005,DHK2011}, with $f$, $X$ and $\Phi$ regarded as $0$-, $1$- and $2$-forms on $G$.
We equip these spaces with the standard inner products
\[\left\{\begin{array}{l}
 \langle f,g\rangle_{L^2(V(G))} =\sum_{v\in V(G)}f(v)g(v),  \\[2mm]
   \langle X,Y\rangle_{L^2_{\wedge}(E(G))} =\sum_{(v_i,v_j)\in E(G)}X(v_i,v_j)Y(v_i,v_j),\\[2mm]
   \langle \Phi,\Psi\rangle_{L^2_{\wedge}(T(G))} =\sum_{(v_i,v_j,v_k)\in T(G)}\Phi(v_i,v_j,v_k)\Psi(v_i,v_j,v_k),
\end{array}\right.\]
for any $f,g\in L^2(V(G))$, $X,Y\in L^2_{\wedge}(E(G))$ and $\Phi,\Psi\in L^2_{\wedge}(T(G))$.

We now introduce the discrete differential operators that generalize gradient, curl and divergence from vector calculus; a detailed exposition can be found in \cite{LHL2020}.
The \emph{graph gradient} $\operatorname{grad} :   L^2(V(G))\to L^2_{\wedge}(E(G))$ is defined by
\[
(\operatorname{grad}f)(v_i,v_j)=
\begin{cases}
f(v_j)-f(v_i), & \text{if } \{v_i,v_j\}\in E(G),\\
0, & \text{otherwise}.
\end{cases}
\]
The \emph{graph curl} $\operatorname{curl} :   L^2_{\wedge}(E(G))\to L^2_{\wedge}(T(G))$ is defined by
\[
(\operatorname{curl}X)(v_i,v_j,v_k)=
\begin{cases}
X(v_i,v_j)+X(v_j,v_k)+X(v_k,v_i), & \text{if } \{v_i,v_j,v_k\}\in T(G),\\
0, & \text{otherwise}.
\end{cases}
\]
The \emph{graph divergence} $\operatorname{div} :   L^2_{\wedge}(E(G))\to L^2(V(G))$ is given by
\[
(\operatorname{div}X)(v_i)=\sum_{v_j\in V(G)}X(v_i,v_j),
\]
where the sum is over all vertices, but the alternating property of $X$ ensures that only neighbors contribute.

Associated with these discrete differential operators are two fundamental Hodge-type operators.
The first is the \emph{graph Laplacian} $\Delta_0 :  L^2(V(G))\to L^2(V(G))$, defined by
\[
\Delta_0 = -\operatorname{div}\operatorname{grad},
\]
which serves as a graph-theoretic analogue of the Laplace operator (cf.\ \cite[Lemma~5.6]{LHL2020}).
The second is the \emph{graph Helmholtzian} \cite{JLYY2011},
\[
\Delta_1 = -\operatorname{grad}\operatorname{div} + \operatorname{curl}^*\operatorname{curl},
\]
which is a discrete analogue of the vector Laplacian.
Here $\operatorname{curl}^* :  L^2_{\wedge}(T(G))\to L^2_{\wedge}(E(G))$ is the adjoint of $\operatorname{curl}$ with respect to the inner products on $L^2_{\wedge}(E(G))$ and $L^2_{\wedge}(T(G))$.
Using the adjoint relation
\[
\langle \operatorname{curl}X,\Phi\rangle_{L^2_{\wedge}(T(G))}
= \langle X,\operatorname{curl}^*\Phi\rangle_{L^2_{\wedge}(E(G))},
\]
one obtains the explicit formula
\[
(\operatorname{curl}^*\Phi)(v_i,v_j)=
\begin{cases}
\displaystyle\sum_{v_k\in V(G)}\Phi(v_i,v_j,v_k), & \text{if } \{v_i,v_j\}\in E(G),\\
0, & \text{otherwise},
\end{cases}\]
for any $\Phi\in L^2_{\wedge}(T(G))$.

For a linear operator $\mathcal{A}$, let $\mathbf{s}(\mathcal{A})$ denote the weakly decreasing rearrangement of its eigenvalues (counting multiplicities).
The nilpotency conditions $\operatorname{curl}\operatorname{grad}=0$ and $\operatorname{div}\operatorname{curl}^*=0$ (see \cite[Theorem~5.7]{LHL2020}) imply
\[
(-\operatorname{grad}\operatorname{div})\circ(\operatorname{curl}^*\operatorname{curl}) = 0 = (\operatorname{curl}^*\operatorname{curl})\circ(-\operatorname{grad}\operatorname{div}).
\]
Consequently, the nonzero spectrum of $\Delta_1$ is the multiset union of the nonzero spectra of $-\operatorname{grad}\operatorname{div}$ and $\operatorname{curl}^*\operatorname{curl}$,
\begin{equation}\label{eq-01}
\mathbf{s}(\Delta_1) \stackrel{\circ}{=} \mathbf{s}(-\operatorname{grad}\operatorname{div}) \uplus \mathbf{s}(\operatorname{curl}^*\operatorname{curl}),
\end{equation}
where $\uplus$ denotes multiset union and $\stackrel{\circ}{=}$ means equality up to zero eigenvalues.
Moreover, by standard linear algebra, we get
\begin{equation}\label{eq-02}
\mathbf{s}(-\operatorname{grad}\operatorname{div}) \stackrel{\circ}{=} \mathbf{s}(-\operatorname{div}\operatorname{grad}) = \mathbf{s}(\Delta_0),
\end{equation}
and the operators $\operatorname{curl}^*\operatorname{curl}$, $-\operatorname{grad}\operatorname{div}$ and $\Delta_0$ are all self-adjoint and positive semidefinite; hence all eigenvalues involved are real and nonnegative.

Let $\mathbf{x}=(x_1,\ldots,x_n)$ and $\mathbf{y}=(y_1,\ldots,y_n)$ be two non-negative sequences which are ordered non-increasing. We say $\mathbf{x}$ is \emph{majorized} by $\mathbf{y}$, denoted $\mathbf{x}\prec \mathbf{y}$, if $\sum_{i=1}^kx_i\le \sum_{i=1}^ky_i$ for $1\le k\le n$ and $\sum_{i=1}^nx_i=\sum_{i=1}^ny_i$. The \emph{conjugate sequence} of $\mathbf{x}$, denoted by $\mathbf{x}^\top=(x_1^\top,\cdots,x_n^\top)$, is defined by $x_i^\top=|\{j\mid x_j\ge i\}|$. In 1994, Grone and Merris \cite{GM1994} conjectured that the eigenvalues of the graph Laplacian $\Delta_0$ are majorized by the conjugate of the vertex degree sequence; this was proved by Bai \cite{B2011} in 2011.
Duval and Reiner \cite{DR2002} proposed a far-reaching analogue for simplicial complexes, formulated in terms of boundary operators.
\begin{conjecture}[{\cite[Conjecture~1.2]{DR2002}}]\label{conj-1}
Let $K$ be a $k$-family on vertex set $V(K)$ and let $\mathbf{s}=(s_1\ge s_2\ge\cdots)$ be the nonincreasing sequence of nonzero eigenvalues of $\partial_{k-1}\partial_{k-1}^*$. Let $\mathbf{d}=(d_1\ge d_2\ge\cdots)$ be the vertex degree sequence of $K$.
Then $\mathbf{s}\prec \mathbf{d}^\top$, where $\prec$ denotes majorization.
\end{conjecture}
For a graph $G$, the edge set $E(G)$ is exactly a $2$-family. Therefore, $\partial_{1}\partial_{1}^{*}$ coincides with the graph Laplacian $\Delta_0 = -\operatorname{div}\operatorname{grad}$, thereby recovering the Grone--Merris conjecture.
Also, the triangle set $T(G)$ is a $3$-families, and the operator $\partial_{2}\partial_{2}^{*}$ for such family coincides with $\operatorname{curl}^{*}\operatorname{curl}$. Thus, to make progress on Conjecture~\ref{conj-1} beyond the case $k=2$, we focus on the operator $\operatorname{curl}^*\operatorname{curl}$ and, more generally, on the graph Helmholtzian $\Delta_1$.
We establish a sharp upper bound for the sum of the two largest eigenvalues of $\operatorname{curl}^*\operatorname{curl}$, and then transfer this bound to $\Delta_1$.

For a vertex $v\in V(G)$, let $\deg_G(v)$ be the standard degree (number of incident edges) and $\deg_G^{(2)}(v)$ be the second-order degree, i.e., the number of triangles containing $v$.
Denote the conjugate sequences of the vertex degree sequence and the second-order degree sequence by $d^\top(G)$ and $D^\top(G)$, respectively. Explicitly, for $i\ge 1$,
\[
d_i^\top(G) = \bigl|\{v\in V(G): \deg_G(v)\ge i\}\bigr|,\qquad
D_i^\top(G) = \bigl|\{v\in V(G): \deg_G^{(2)}(v)\ge i\}\bigr|.
\]

Our first main result is the following majorization-type inequality for $\operatorname{curl}^*\operatorname{curl}$.
\begin{theorem}\label{thm-main-1}
For any graph $G$, we have
\begin{equation}\label{eq-m-1}
 \lambda_1\left(\operatorname{curl}^*(G)\operatorname{curl}(G)\right) + \lambda_2(\operatorname{curl}^*(G)\operatorname{curl}(G)) \le D_1^\top(G) + D_2^\top(G),   
\end{equation}
where $\lambda_i\left(\operatorname{curl}^*(G)\operatorname{curl}(G)\right)$ denotes the $i$-th largest eigenvalue of $\operatorname{curl}^*\operatorname{curl}$.
\end{theorem}

Combining \eqref{eq-01} and \eqref{eq-02}, we have
\begin{equation}\label{eq-l-a-1}
  \mathbf{s}(\Delta_1) \stackrel{\circ}{=} \mathbf{s}(\Delta_0) \uplus \mathbf{s}(\operatorname{curl}^*\operatorname{curl}).  
\end{equation}
Moreover, \[d_1^\top(G)\ge d_2^\top(G)\ge D_1^\top(G) \ge D_2^\top(G).\]
This spectral decomposition, together with the above inequalities,  allows us to control the top eigenvalues of the Helmholtzian.
\begin{corollary}\label{thm-main-2}
Let $G$ be a graph. Then the sum of the two largest eigenvalues of $\Delta_1$ satisfies the following bounds.
\begin{enumerate}
\item[(i)] If both $\lambda_1(\Delta_1)$ and $\lambda_2(\Delta_1)$ belong to $\mathbf{s}(\operatorname{curl}^*\operatorname{curl})$, then
\[
\lambda_1(\Delta_1)+\lambda_2(\Delta_1) \le D_1^\top(G)+D_2^\top(G).
\]
\item[(ii)] If neither $\lambda_1(\Delta_1)$ nor $\lambda_2(\Delta_1)$ belongs to $\mathbf{s}(\operatorname{curl}^*\operatorname{curl})$, then
\[
\lambda_1(\Delta_1)+\lambda_2(\Delta_1) \le d_1^\top(G)+d_2^\top(G).
\]
\item[(iii)] Otherwise, $\lambda_1(\Delta_1)+\lambda_2(\Delta_1) \le d_1^\top(G)+D_1^\top(G).$
\end{enumerate}
\end{corollary}

The above estimates are obtained for graphs, but the result for
$\operatorname{curl}^*\operatorname{curl}$ yields a corresponding bound for the up-Laplacian of a general $3$-family, which need not be realizable as the family of triangles of a graph.
Let $K$ be a $3$-family on vertex set $V(K)$. For any $i\in V(K)$, its \emph{degree} $\operatorname{deg}_K(i)$ is the number of $3$-sets in $K$ that contain $i$. Let $\mathbf{d}(K)=(d_1(K)\ge d_2(K)\ge\cdots)$ be the vertex degree sequence of $K$, and let $\mathbf{d}^{\top}(K)=(d_1^\top(K)\ge d_2^\top(K)\ge\cdots)$ denote its conjugate degree sequence.
\begin{corollary}\label{cor-simplicial}
Let $K$ be a $3$-family on vertex set $V(K)$. Then
\[
\lambda_1(\partial_{2}(K)\partial_{2}^*(K)) + \lambda_2(\partial_{2}(K)\partial_{2}^*(K)) \le d_1^\top(K) + d_2^\top(K),
\]
where $\partial_{2}(K)\partial_{2}^*(K)$ denotes the up-Laplacian of $K$.
\end{corollary}

\section{Preliminaries}

We begin by recalling a classical result on eigenvalues of symmetric matrices that will be used repeatedly.

\begin{lemma}[Ky Fan's inequality {\cite[Proposition~A.6]{MOA2011}}]\label{thm-5}
Let $A,B\in\mathbb{R}^{n\times n}$ be symmetric matrices, and let $1\le k\le n$. Then
\[
\sum_{i=1}^k \lambda_i(A+B) \le \sum_{i=1}^k \lambda_i(A) + \sum_{i=1}^k \lambda_i(B).
\]
\end{lemma}

\subsection{Oriented incidence structure and matrix representations}

We now give the matrix representations of the operators $-\operatorname{grad}\operatorname{div}$, $\operatorname{curl}^{*}\operatorname{curl}$, and $\Delta_1$.
Let $G$ be a simple graph with vertex set $V(G)$, edge set $E(G)$, and triangle set $T(G)$.
We fix arbitrary orientations on all edges and on all triangles.

For an oriented edge $e$, we denote its tail and head by $e^{-}$ and $e^{+}$, respectively, and write $V(e)=\{e^{-},e^{+}\}$.
If there is an oriented edge from $u$ to $v$, we write $u\to v$.
If $u$ and $v$ are adjacent, we write $u\sim v$; otherwise $u\nsim v$.
For a vertex $v$ and an edge $e$, we write $v\in e$ if $v\in V(e)$, and we write $v\to e$ if $v=e^{-}$, $e\to v$ if $v=e^{+}$.

For two edges $e_1,e_2\in E(G)$, we write $e_1\sim e_2$ if $V(e_1)\cap V(e_2)\neq\varnothing$.
Furthermore:
\begin{itemize}
\setlength{\itemsep}{-1pt}
    \item $e_1\to e_2$ if $e_1^{+}=e_2^{-}$;
    \item $e_1\stackrel{+}{\sim}e_2$ if $e_1^{+}=e_2^{+}$, and $e_1\stackrel{-}{\sim}e_2$ if $e_1^{-}=e_2^{-}$;
    \item $e_1\stackrel{\pm}{\sim}e_2$ if either $e_1\stackrel{+}{\sim}e_2$ or $e_1\stackrel{-}{\sim}e_2$;
    \item $e_1\leftrightarrow e_2$ if either $e_1\to e_2$ or $e_2\to e_1$;
    \item $e_1\triangle e_2$ if $e_1$ and $e_2$ belong to a common triangle.
\end{itemize}
For an edge $e$ and a triangle $\triangle$, we write $e\in\triangle$ if $e$ is an edge of $\triangle$.
If the orientation of $e$ is compatible with that of $\triangle$, we write $e\in\triangle^{+}$; otherwise $e\in\triangle^{-}$.
The \emph{triangle degree} of an edge $e\in E(G)$ is defined by
$\deg_G(e) = \bigl|\{\triangle\in T(G) \mid e\in\triangle\}\bigr|$.

For clarity, we summarize the above conventions in Table~\ref{tab-a-1}.
\begin{table}[!ht]
\centering
\caption{Relations between vertices, edges, and triangles. 
(a) Vertex--vertex, vertex--edge, and common-triangle relations. 
(b) Edge--edge incidence and edge--triangle relations.}
\label{tab-a-1}
\begin{minipage}[t]{0.48\textwidth}
\centering
\begin{tabular}{|c|c|c|}
\hline
Symbol & Relation & Diagram \\
\hline
\multirow{2}{*}{$u\sim v$} & $u\rightarrow v$ & 
  \includegraphics[scale=0.45]{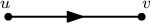} \\[1.2mm]
\cline{2-3}
& $v\rightarrow u$ & 
  \includegraphics[scale=0.45]{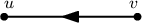} \\[1.2mm]
\hline
\multirow{2}{*}{$u\in e$} & $u\rightarrow e$ ($u=e^-$) & 
  \includegraphics[scale=0.45]{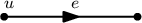} \\[1.2mm]
\cline{2-3}
& $e\rightarrow u$ ($u=e^+$) & 
  \includegraphics[scale=0.45]{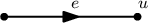} \\[1.2mm]
\hline
\multicolumn{2}{|c|}{$e_1\triangle e_2$} & 
  \raisebox{-.5\height}{\includegraphics[scale=0.45]{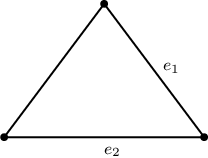}} \\[1.2mm]
\hline
\end{tabular}
\end{minipage}
\hfill
\begin{minipage}[t]{0.48\textwidth}
\centering
\begin{tabular}{|c|c|c|}
\hline
Symbol & Relation & Diagram \\
\hline
\multirow{4}{*}{$e_1\sim e_2$} & 
  \multirow{2}{*}{$e_1\overset{\pm}{\sim}e_2$} & 
  $e_1\overset{+}{\sim}e_2$\; \includegraphics[scale=0.45]{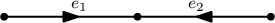} \\[1.2mm]
\cline{3-3}
& & $e_1\overset{-}{\sim}e_2$\; \includegraphics[scale=0.45]{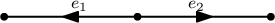} \\[1.2mm]
\cline{2-3}
& \multirow{2}{*}{$e_1\leftrightarrow e_2$} & 
  $e_1\rightarrow e_2$\; \includegraphics[scale=0.45]{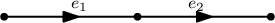} \\[1.2mm]
\cline{3-3}
& & $e_2\rightarrow e_1$\; \includegraphics[scale=0.45]{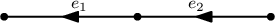} \\[1.2mm]
\hline
\multirow{2}{*}{$e\in\triangle$} & $e\in\triangle^+$ & 
  \raisebox{-.5\height}{\includegraphics[scale=0.45]{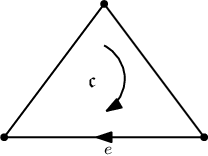}} \\[1.2mm]
\cline{2-3}
& $e\in\triangle^-$ & 
  \raisebox{-.5\height}{\includegraphics[scale=0.45]{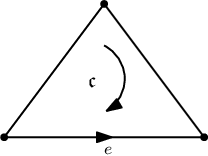}} \\[1.2mm]
\hline
\end{tabular}
\end{minipage}
\end{table}

The \emph{edge--vertex incidence matrix} of $G$, denoted by $\mathcal{B}=(\mathcal{B}_{ev})$, is the matrix with rows indexed by $E(G)$ and columns by $V(G)$.
The \emph{triangle--edge incidence matrix}, denoted by $\mathcal{C}=(\mathcal{C}_{\triangle e})$, is the matrix with rows indexed by $T(G)$ and columns by $E(G)$.
They are defined by
\[
\mathcal{B}_{ev}=
\begin{cases}
-1, & \text{if } v\to e,\\
1,  & \text{if } e\to v,\\
0,  & \text{otherwise},
\end{cases}
~\text{and}~
\mathcal{C}_{\triangle e}=
\begin{cases}
-1, & \text{if } e\in\triangle^{-},\\
1,  & \text{if } e\in\triangle^{+},\\
0,  & \text{otherwise}.
\end{cases}
\]

According to \cite{LHL2020}, the matrix representations of $-\operatorname{grad}\operatorname{div}$ and $\operatorname{curl}^{*}\operatorname{curl}$ are given by $\mathcal{B}\mathcal{B}^\top$ and $\mathcal{C}^\top\mathcal{C}$, respectively.
Their entries are described explicitly below.

\begin{lemma}[\cite{LHL2020}]\label{lem-1}
Let $G$ be a graph with fixed orientations on $E(G)$ and $T(G)$. Then the matrix representing $-\operatorname{grad}\operatorname{div}$ is $\mathcal{B}\mathcal{B}^\top$, indexed by $E(G)$, with entries
\[
(\mathcal{B}\mathcal{B}^\top)_{ee'}=
\begin{cases}
2,    & \text{if } e' = e,\\[2pt]
-1,   & \text{if } e \leftrightarrow e',\\[2pt]
1,    & \text{if } e' \overset{\pm}{\sim} e,\\[2pt]
0,    & \text{otherwise}.
\end{cases}
\]
\end{lemma}

\begin{lemma}[\cite{LHL2020}]\label{lem-2}
Under the same assumptions, the matrix representing $\operatorname{curl}^{*}\operatorname{curl}$ is $\mathcal{C}^\top\mathcal{C}$, indexed by $E(G)$, with entries
\[
(\mathcal{C}^\top\mathcal{C})_{ee'}=
\begin{cases}
\deg_G(e), & \text{if } e' = e,\\[2pt]
1,         & \text{if } e \leftrightarrow e' \text{ and } e \triangle e',\\[2pt]
-1,        & \text{if } e' \overset{\pm}{\sim} e \text{ and } e \triangle e',\\[2pt]
0,         & \text{otherwise}.
\end{cases}
\]
\end{lemma}

\begin{lemma}\label{lem-3}
The spectrum of $\operatorname{curl}^{*}\operatorname{curl}$ is independent of the choice of orientations on $E(G)$ and $T(G)$.
\end{lemma}
\begin{proof}
Let $\gamma_1$ and $\gamma_2$ be two distinct choices of orientations on $E(G)$ and $T(G)$, and let $\mathcal{C}$ and $\mathcal{C}^{\prime}$ denote the corresponding triangle–edge incidence matrices associated with these orientations, respectively. Let $E$ be a diagonal sign matrix indexed by $E(G)$, defined by
\[
E_{ee} =
\begin{cases}
	1, & \text{if the orientation of } e \text{ is the same under } \gamma_1 \text{ and } \gamma_2,\\
	-1, & \text{otherwise}.
\end{cases}
\]
As shown in Lemma~\ref{lem-2}, the entrywise structure of the incidence relation between edges and triangles is invariant under simultaneous preservation or reversal of the orientations of $e$ and $e'$. In all other cases, the only effect of changing orientations is an interchange of signs in the cases $e \leftrightarrow e'$ and $e' \overset{\pm}{\sim} e$, while all other relations remain unchanged.
Consequently, by Lemma~\ref{lem-2}, we obtain
\[
\left(\mathcal{C}^{\prime\top}\mathcal{C}^\prime\right)_{ee'}  = E_{e,e}\left(\mathcal{C}^{\top}\mathcal{C}\right)_{ee'}E_{e^{\prime}e^{\prime}},
\]
and hence $\mathcal{C}^{\prime\top}\mathcal{C}^\prime= E^{\top}\left(\mathcal{C}^{\top}\mathcal{C} \right)E.$
Since $E$ is a diagonal matrix with entries $\pm 1$, we have $E^{-1}=E^{\top}=E$. Therefore,
\[
\mathcal{C}'^{\top}\mathcal{C}' = E^{-1}(\mathcal{C}^{\top}\mathcal{C})E,
\]
so $\mathcal{C}'^{\top}\mathcal{C}'$ and $\mathcal{C}^{\top}\mathcal{C}$ are similar matrices. Consequently, they have the same spectrum, which completes the proof.
\end{proof}

Since $\Delta_1 = -\operatorname{grad}\operatorname{div} + \operatorname{curl}^{*}\operatorname{curl}$, its matrix representation follows immediately.

\begin{lemma}\label{thm-133}
Let $G$ be a graph with orientations on $E(G)$ and $T(G)$. Then the graph Helmholtzian is represented by the matrix $\mathcal{H}(G)=(h_{ee'})=\mathcal{B}\mathcal{B}^{\top}+\mathcal{C}^{\top}\mathcal{C}$, indexed by $E(G)$, where
\[
h_{ee'}=
\begin{cases}
\deg_G(e)+2, & \text{if } e' = e,\\[2pt]
-1,          & \text{if } e \leftrightarrow e' \text{ and } e \ntriangle e',\\[2pt]
1,           & \text{if } e' \overset{\pm}{\sim} e \text{ and } e \ntriangle e',\\[2pt]
0,           & \text{otherwise}.
\end{cases}
\]
\end{lemma}

\subsection{Known bounds for the largest eigenvalue}

We also record three known estimates for the largest eigenvalue of $\operatorname{curl}^{*}\operatorname{curl}$ that follow from earlier work on simplicial complexes.

\begin{lemma}[{\cite{DR2002}}]\label{thm-dr02-2}
Let $G$ be a graph with vertex set $V(G)$, edge set $E(G)$, and triangle set $T(G)$. Then
\[
\lambda_1(\operatorname{curl}^{*}(G)\operatorname{curl}(G)) \le D_1^\top(G).
\]
\end{lemma}

\begin{lemma}[{\cite{DR2002}}]\label{thm-dr02-1}
Let $G$ be a graph with vertex set $V(G)$, edge set $E(G)$, and triangle set $T(G)$. Then
\[
\lambda_1(\operatorname{curl}^{*}(G)\operatorname{curl}(G)) \ge \max\{\deg_G(e) \mid e\in E(G)\} + 2.
\]
\end{lemma}

\begin{lemma}[\cite{FWW2024}]\label{thm-11}
Let $G$ be a graph with vertex set $V(G)$, edge set $E(G)$, and triangle set $T(G)$. Then
\[
\lambda_1(\operatorname{curl}^{*}(G)\operatorname{curl}(G)) \le \max\Bigl\{\,\sum_{e\in\triangle}\deg_G(e) \mid \triangle\in T(G)\Bigr\}.
\]
\end{lemma}

\section{Proofs of the main results}\label{se-3}
Before proceeding with the proof, we fix some notation and standing assumptions that will be used throughout this section. For two vertex disjoint graphs $G_1$ and $G_2$, denote by $G_1\cup G_2$ the \emph{union} of them. That is, $V(G_1\cup G_2)=V(G_1)\cup V(G_2)$ and $E(G_1\cup G_2)=E(G_1)\cup E(G_2)$. Their \emph{join} $G_1\vee G_2$ is the graph obtained from $G_1\cup G_2$ by adding all possible edges between $G_1$ and $G_2$. The union of $t$ copies of a graph $H$ is denoted by $tH$. As usual, we use $K_n$, $P_n$, and $C_n$ to denote the complete graph, the path, and the cycle on $n$ vertices, respectively.

 Let $G$ be a graph with vertex set $V(G)$, edge set $E(G)$, and triangle set $T(G)$. Recall that, for a vertex $v\in V(G)$, its second-order degree is defined as $\deg_G^{2}(v)=|\{\triangle\mid v\in \triangle\}|$. Let $V_1(G) = \{v \in V(G) \mid \deg_G^{(2)}(v) = 1\}$ and $V_2(G) = \{v \in V(G)\mid \deg_G^{(2)}(v) \ge 2\}$. For a positive integer $i$, denote by $D_i^{\top}(G)=\{v\in V(G)\mid \deg_G^{(2)}(v)\ge i\}$. Also, for an edge $e\in E(G)$, denote by $\deg_G(e)$ the number of triangles containing $e$. For positive integers $n,k,a_1,\ldots,a_k$, denote by $\mathcal{G}_{n;k}^{a_1,\ldots,a_k}$ the set of graphs obtained from $K_n$ by choosing $k$ distinct edges, and then, for each $1\le i\le k$, attaching $a_i$ pairwise nonadjacent new vertices, each of which is joined to both endpoints of the $i$-th chosen edge.

\subsection{Proof of Theorem \ref{thm-main-1}}
Note that if $\deg_G^{(2)}(v)=0$ for some vertex $v\in V(G)$, then every edge incident with $v$ does not belong to any triangle. Consequently, the corresponding rows and columns in the matrix representation $\mathcal{C}^\top(G)\mathcal{C}(G)$ are zero, and hence such edges contribute only to the multiplicity of the eigenvalue $0$. 
\textbf{Therefore, without loss of generality, we restrict our attention to graphs in which every edge (as well as every vertex) is contained in at least one triangle.} Therefore, for a graph $G$, we may assume that
$D_1^\top(G)=|V(G)|$, $V(G)=V_1(G)\sqcup V_2(G)$, and $D_2^\top(G) = |V_2(G)|$. In what follows, we show that Eq.~\eqref{eq-m-1} always holds.

According to Lemma \ref{thm-dr02-2}, we have \begin{equation}\label{eq-dr-1}
    \lambda_1\left(\operatorname{curl}^*(G)\operatorname{curl}(G)\right)\le D_1^{\top}(G)
\end{equation}
If $\lambda_2(\operatorname{curl}^*(G)\operatorname{curl}(G))=0$, then we obtain
\[\lambda_1\left(\operatorname{curl}^*(G)\operatorname{curl}(G)\right)
+\lambda_2\left(\operatorname{curl}^*(G)\operatorname{curl}(G)\right)
\le D_1^{\top}(G),\]
and \eqref{eq-m-1} follows immediately. Therefore, without loss of generality, we may assume that $\operatorname{curl}^*(G)\operatorname{curl}(G)$ has at least two nonzero eigenvalues. This in turn leads to $t(G)\ge 2$.

\begin{figure}[htbp]
	\centering
	\includegraphics[width=16cm]{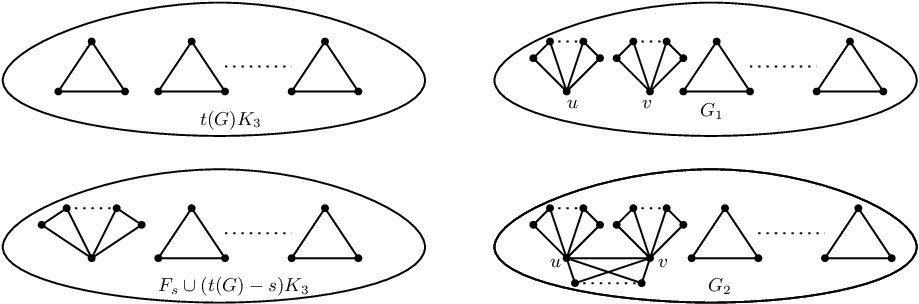}
	\caption{\small The possibilities of $G$ when $D_2^{\top}(G)\le 2$.}
	\label{f-thm-1}
\end{figure}
First, we consider the case for $D_2^{\top}(G)\le 2$. If $D_2^{\top}(G)=0$, then $G\cong t(G)K_3$. If $D_2^{\top}(G)=1$, then $G\cong F_s\cup (t(G)-s)K_3$ for some $s\ge 2$, where $F_s=K_1\vee (sK_2)$ is the friendship graph. If $D_2^{\top}(G)=2$, then assume $V_2(G)=\{u,v\}$. Therefore, we conclude that either $G\cong G_1$ when $u\not\sim v$, or $G\cong G_2$ when $u\sim v$ (shown in Fig. \ref{f-thm-1}). One can verify such cases by directed calculations. We omit the calculations here but put them in Appendix \ref{app-1}.

Next, we consider the case for $D_2^{\top}(G)\ge 3$. We divide the following two cases to discuss.

\vspace{10pt}
\noindent\textbf{Case 1.} There exists a triangle $\triangle \in T(G)$ such that $\left|\triangle\cap V_1(G)\right|\ge 2$.
\vspace{10pt}

We proceed by induction on $t(G)$. Since $D_2^{\top}(G)\ge 3$, we have $t(G)\ge 3$. If $t(G)=3$, then $G\cong G_4$ (see Figure \ref{f-thm-4}). 
\begin{figure}[htbp]
	\centering
	\includegraphics[width=7cm]{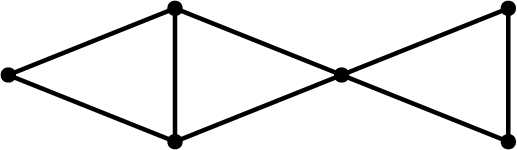}
	\vspace{3mm}
	\caption{The graph $G_4$ in Case 4.}
	\label{f-thm-4}
\end{figure}
Then $\mathbf{s}\left(\operatorname{curl}^*(G)\operatorname{curl}(G)\right)\stackrel{\circ}{=}\left\{4,3,2\right\}$, and thus we have $\lambda_2\left(\operatorname{curl}^*(G)\operatorname{curl}(G)\right)=3\le D_2^{\top}(G)$, and \eqref{eq-m-1} follows.
Suppose now that the induction hypothesis holds for all graphs with fewer than $t(G)$ triangles. Choose a triangle $\triangle\in T(G)$ such that
 $\left|\triangle\cap V_1(G)\right|\ge 2$. Let $e$ be an edge of $\triangle$ with both endpoints in $V_1(G)$, and let $G_0=G-e$. Then $\triangle$ shares no edge with any other triangle. Therefore, one can easily verify that
 \[\mathbf{s}(\operatorname{curl}^*(G)~\operatorname{curl}(G))=\mathbf{s}(\operatorname{curl}^*(G_0)~\operatorname{curl}(G_0))\uplus \{3\}.\]
If $\lambda_2(\operatorname{curl}^*(G)~\operatorname{curl}(G))\le 3$, then \eqref{eq-m-1} holds immediately due to \eqref{eq-dr-1} and the assumption $D_2^{\top}(G)\ge 3$. If $\lambda_2(\operatorname{curl}^*(G)~\operatorname{curl}(G))\ge 4$, then
\begin{equation}\label{eq-thm-2}
	\lambda_i\!\left(\operatorname{curl}^*(G_0)\operatorname{curl}(G_0)\right)=\lambda_i\!\left(\operatorname{curl}^*(G)\operatorname{curl}(G)\right)\quad\text{for} \quad i\in\{1,2\}.
\end{equation}
Moreover, if $|\triangle\cap V_1(G)|=3$, then $D_1^{\top}(G_0)=D_1^{\top}(G)-3$ and $D_2^{\top}(G_0)=D_2^{\top}(G)$; if $|\triangle\cap V_1(G)|=2$, then either $D_1^\top\left(G_0\right)=D_1^\top\left(G\right)-2$ and $D_2^\top\left(G_0\right)=D_2^\top\left(G\right)$,
or $D_1^\top\left(G_0\right)=D_1^\top\left(G\right)-1$ and $D_2^\top\left(G_0\right)=D_2^\top\left(G\right)-1.$ In all cases, we always have 
\begin{equation}\label{eq-a-1}
   D_1^{\top}(G_0)+D_2^{\top}(G_0)\le D_1^{\top}(G)+D_2^{\top}(G)-2. 
\end{equation}
By the inductive assumption, combining \eqref{eq-thm-2} and \eqref{eq-a-1}, we get
\[\begin{array}{ll}
    &\lambda_1(\operatorname{curl}^*(G)\operatorname{curl}(G))+\lambda_2(\operatorname{curl}^*(G)\operatorname{curl}(G))\\[2mm]
    =&  \lambda_1(\operatorname{curl}^*(G_0)\operatorname{curl}(G_0))+\lambda_2(\operatorname{curl}^*(G_0)\operatorname{curl}(G_0))\\[2mm]
    \le&D_1^{\top}(G_0)+D_2^{\top}(G_0)\le D_1^{\top}(G)+D_2^{\top}(G)-2,
\end{array}\]
as desired.

\vspace{10pt}
\noindent\textbf{Case 2.} $|\triangle\cap V_1(G)|\le 1$ for every $\triangle\in T(G)$. 
\vspace{10pt}

If $V_1(G)=\emptyset$, then we have
\[D_1^\top(G) = |V(G)|=|V_1(G)|+|V_2(G)|=|V_2(G)|=D_2^\top(G).\]
Thus, \eqref{eq-dr-1} indicates
\[\begin{array}{ll}
     &\lambda_1\left(\operatorname{curl}^{*}(G)\operatorname{curl}(G)\right)+ \lambda_2\left(\operatorname{curl}^{*}(G)\operatorname{curl}(G)\right)  \\[2mm]
    \le &  2\lambda_1\left(\operatorname{curl}^{*}(G)\operatorname{curl}(G)\right) \le 2D_1^\top(G)= D_1^\top(G)+D_2^\top(G),
\end{array}\]
as desired. 

If $V_1(G)\ne\emptyset$, then for every $v\in V_1(G)$, there exists a unique edge $e$ in $G[V_2(G)]$ such that $e\cup\{v\}$ forms a triangle, i.e., $e\cup\{v\}\in T(G)$. For an edge $e\in G[V_2(G)]$, denote by
\[V_{1,e}=\{v\in V_1(G)\mid e\cup v\in T(G)\}.\]
Assume that $e_1,\ldots,e_k$ are all the edges in $G[V_2(G)]$ with $V_{1,e_i}\ne\emptyset$. Let $a_i=|V_{1,e_i}|$ for $1\le i\le k$. We may assume that $a_1\ge a_2\ge\cdots\ge a_k\ge 1$. Thus, $V_1(G)=\cup_{1\le i\le k}V_{1,e_i}$ and $G$ has the structure as shown in Figure~\ref{f-X}.
\begin{figure}[htbp]
	\centering
	\includegraphics[width=12cm]{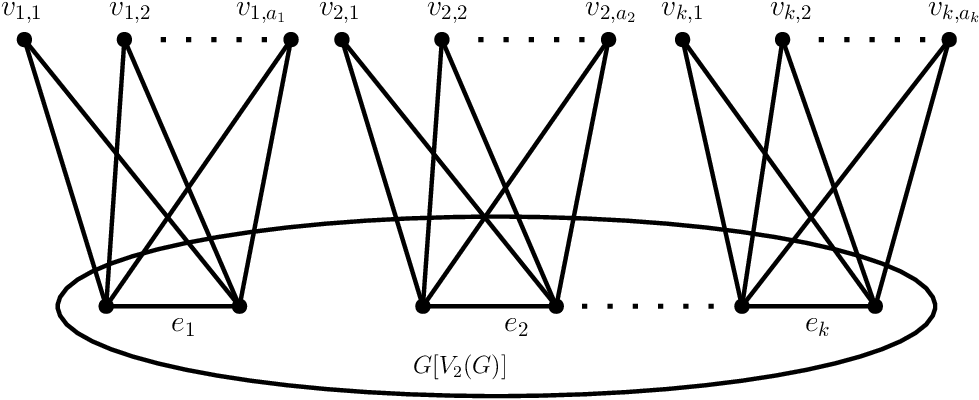}
\caption{The structure of $G$ (the edges in $V_2(G)$ are not shown).}
	\label{f-X}
\end{figure}
For convenience, denote by $T=T(G)$, $T_1=\{\triangle\in T(G)\mid \triangle\cap V_1(G)\ne\emptyset\}$, and $T_2=\{\triangle\in T(G)\mid \triangle\cap V_1(G)=\emptyset\}$.
Consider the $3$-families $T$, $T_1$ and $T_2$. Then from $T=T_1\sqcup T_2$ and the definition of $\partial_{2}\partial_{2}^*$ (see \cite{DR2002}), it follows that
 \begin{equation}\label{eq-l-00}
 	\partial_{2}(T)\partial_{2}^*(T)=\partial_{2}(T_1)\partial_{2}^*(T_1)+\partial_{2}(T_2)\partial_{2}^*(T_2),
 \end{equation}
 and that the $\partial_{2}(T)\partial_{2}^{*}(T)=\operatorname{curl}^{*}(G)\operatorname{curl}(G)$ . 
Therefore, combining Eq.~\eqref{eq-l-00} and Lemma \ref{thm-5}, we get
\begin{equation}\label{eq-l-1}
	\begin{array}{ll}
	   &\lambda_1\left(\operatorname{curl}^{*}(G)\operatorname{curl}(G)\right)+ \lambda_2\left(\operatorname{curl}^{*}(G)\operatorname{curl}(G)\right)\\[2mm]
       \le&\lambda_1\left(\partial_{2}(T_1)\partial_{2}^*(T_1)\right)+ \lambda_2\left(\partial_{2}(T_1)\partial_{2}^*(T_1)\right)+\lambda_1\left(\partial_{2}(T_2)\partial_{2}^*(T_2)\right) + \lambda_2\left(\partial_{2}(T_2)\partial_{2}^*(T_2)\right).
	\end{array}
\end{equation}
By noticing that $d_1^{\top}(T_2)=|V(T_2)|\le D_2^{\top}(G)$, from \cite[Proposition 6.2]{DR2002}, we get
\begin{equation}\label{eq-l-2}
	\lambda_1\left(\partial_{2}(T_2)\partial_{2}^*(T_2)\right) + \lambda_2\left(\partial_{2}(T_2)\partial_{2}^*(T_2)\right)\le 2\lambda_1\left(\partial_{2}(T_2)\partial_{2}^*(T_2)\right) \le 2 d_1^{\top}(T_2)\le 2D^\top_2(G).
\end{equation}
For the family $T_1$, the matrix representation of $\partial_{2}^*(T_1)\partial_2(T_1)$ is $\operatorname{diag}(A_1,\ldots,A_k)$ under suitable orientation, where $A_i=2I_{a_i}+J_{a_i}$, and $I$ and $J$ denote the identity matrix and the all-one matrix, respectively. This yields that
\[\mathbf{s}\left(\partial_{2}(T_1)\partial_{2}^*(T_1)\right)\stackrel{\circ}{=}\mathbf{s}\left(\partial_{2}^*(T_1)\partial_{2}(T_1)\right)\stackrel{\circ}{=}\left\{a_i+2, 2^{\left(|V_1(G)|-k\right)}\mid 1\le i\le k\right\},\]
which, combined with the relation \[\sum_{i=1}^{k}a_i=|V_1(G)|=|V(G)|-|V_2(G)|=D_1^{\top}(G)-D_2^{\top}(G),\] yields
\begin{equation}\label{eq-l-3}
	\lambda_1\left(\partial_{2}(T_1)\partial_{2}^*(T_1)\right)+ \lambda_2\left(\partial_{2}(T_1)\partial_{2}^*(T_1)\right)=a_1+a_2+4=D_1^{\top}(G)-D_2^{\top}(G)+4-\sum_{i=3}^{k}a_i.
\end{equation}
Combining \eqref{eq-l-1}, \eqref{eq-l-2} with  \eqref{eq-l-3}, we have
\[\begin{array}{lll}
	&&\lambda_1\left(\operatorname{curl}^{*}(G)\operatorname{curl}(G)\right)+ \lambda_2\left(\operatorname{curl}^{*}(G)\operatorname{curl}(G)\right)\\[2mm]
	&\le& 2D^\top_2(G)+D_1^{\top}(G)-D_2^{\top}(G)+4-\sum_{i=3}^{k}a_i\\[2mm]
	&=&D_1^{\top}(G)+D_2^{\top}(G)+4-\sum_{i=3}^{k}a_i.
\end{array}\]
If $k\ge 6$, then $\sum_{i=3}^k a_i \ge 4$, which immediately implies \eqref{eq-m-1}. Hence, it remains to consider the case $k\le 5$.

Let $|V_2(G)|=n$. By the construction of $G$ shown in Fig.~\ref{f-X}, $G$ is a spanning subgraph of some graph $G^{\prime}\in \mathcal{G}_{n;k}^{a_1,\ldots,a_k}$. One can easily verify that the matrix representation $\mathcal{C}(G)\mathcal{C}^\top(G)$ is a principal submatrix of
$\mathcal{C}(G')\mathcal{C}^\top(G')$. Hence, by the famous interlacing theorem (see, for example, \cite[Chapter 9]{GR2001}),
\begin{equation}\label{eq-l-0}
\lambda_1\!\left(\operatorname{curl}^{*}(G)\operatorname{curl}(G)\right)
+\lambda_2\!\left(\operatorname{curl}^{*}(G)\operatorname{curl}(G)\right)
\le
\lambda_1\!\left(\operatorname{curl}^{*}(G')\operatorname{curl}(G')\right)
+\lambda_2\!\left(\operatorname{curl}^{*}(G')\operatorname{curl}(G')\right).
\end{equation}
Since $D_1^{\top}(G')=D_1^{\top}(G)$ and $D_2^{\top}(G')=D_2^{\top}(G)$, it follows from \eqref{eq-l-0} that it suffices to show that  \[\lambda_1\left(\operatorname{curl}^{*}(G')\operatorname{curl}(G')\right)+ \lambda_2\left(\operatorname{curl}^{*}(G')\operatorname{curl}(G')\right)\le D_1^{\top}(G')+D_2^{\top}(G')\]
holds for any $G^{\prime}\in \mathcal{G}_{n;k}^{a_1,\ldots,a_k}$ and $k\le 5$. This is accomplished by the following lemma.

\begin{lemma}\label{lem-main-1}
	Let $n,k$ be integers with $\binom{n}{2}\ge k$ and $k\le 5$. Then, for every graph $G\in \mathcal{G}_{n;k}^{a_1,\ldots,a_k}$, we have
	\[\lambda_1\left(\operatorname{curl}^{*}(G)\operatorname{curl}(G)\right) +\lambda_2\left(\operatorname{curl}^{*}(G)\operatorname{curl}(G)\right)\le D_1^\top(G)+D_2^\top(G).\]
\end{lemma}
\begin{proof}[\rm\textbf{Proof of Lemma \ref{lem-main-1}}]
In what follows, we provide a detailed proof only for the case $k=5$; the remaining cases are analogous and are therefore omitted. For the graph $G$,  we denote by $e_i$ the $i$-th chosen edge of $K_n$, and by $v_{i,1},\ldots,v_{i,a_i}$ the $a_i$ new vertices attached to $e_i$ in $G$, for each $1 \le i \le 5$.
 It follows from Lemma \ref{lem-3} that the spectrum of $\operatorname{curl}^{*}(G)\operatorname{curl}(G)$ is independent of the choice of orientations on $E(G)$ and $T(G)$. To simplify the subsequent analysis, we fix an orientation $\gamma$ on $E(G)$ and $T(G)$ for $G$ such that, for all $1 \le i \le 5$ and $1 \le j \le a_i$:
\begin{itemize}
\setlength{\itemsep}{-1pt}
\item[(i)] the three edges of each triangle $e_i \cup \{v_{i,j}\}$ are oriented according to the cyclic order;
\item[(ii)] each such triangle inherits the orientation from the cyclic order in (i), as shown in Figure \ref{f-2}.
\end{itemize}

\begin{figure}[htbp]
	\centering
	\includegraphics[width=14cm]{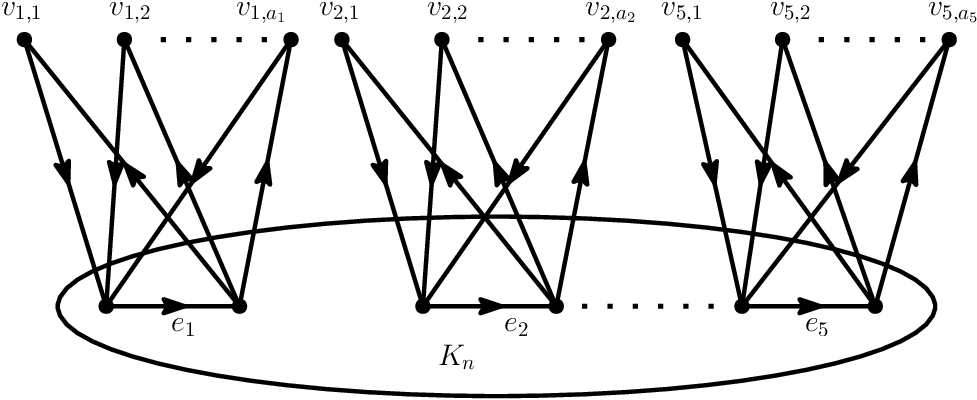}
	\caption{$ G\in \mathcal{G}_{n;5}^{a_1,\ldots,a_5}$, with the adjacency relations among $e_1,\ldots,e_5$ unspecified.}
	\label{f-2}
\end{figure}
 
	Since $D_1^\top(G)=n+\sum_{i=1}^{5} a_i$ and $D_2^\top(G)=n$, it suffices to show that \begin{equation}\label{eq-m-2}
	    \lambda_1\left(\operatorname{curl}^{*}(G)\operatorname{curl}(G)\right) +\lambda_2\left(\operatorname{curl}^{*}(G)\operatorname{curl}(G)\right) \le 2n+\sum_{i=1}^{5} a_i.
	\end{equation}
If $\lambda_2\left(\operatorname{curl}^{*}(G)\operatorname{curl}(G)\right)  \le n= D_2^{\top}(G)$, then we get \eqref{eq-m-2} immediately from Lemma \ref{thm-dr02-2}. \textbf{Therefore, it remains to consider the case for $\lambda_2\left(\operatorname{curl}^{*}(G)\operatorname{curl}(G)\right) > n$, which implies that the operator $\operatorname{curl}^{*}(G)\operatorname{curl}(G)$ has at least two eigenvalues greater than $n$.}

Let $\lambda$ be an eigenvalue of $\operatorname{curl}^{*}(G)\operatorname{curl}(G)$ satisfying $\lambda>n$. In particular, $\lambda$ is also an eigenvalue of its matrix representation $\mathcal{C}^\top (G)\mathcal{C}(G)$. Let $\mathbf{x}=(x_e)_{e\in E(G)} \in \mathbb{R}^{|E(G)|}$ be a corresponding eigenvector. From the eigenvalue equation $\left(\mathcal{C}^\top (G)\mathcal{C}(G)\right)\mathbf{x} = \lambda \mathbf{x}$, it follows that, for all $1\le i\le 5$ and $1\le j\le a_i$,
\[\left\{\begin{matrix}
	x_{\left\{v_{i,j},e^{-}_i\right\}}+x_{\left\{v_{i,j},e^{+}_i\right\}}+x_{e_i}=\lambda x_{\left\{v_{i,j},e^{-}_i\right\}}, \\[2mm]
	x_{\left\{v_{i,j},e^{+}_i\right\}}+x_{\left\{v_{i,j},e^{-}_i\right\}}+x_{e_i}=\lambda x_{\left\{v_{i,j},e^{+}_i\right\}},
\end{matrix}\right.\]
which immediately implies \begin{equation}\label{eq-3}
	x_{\left\{v_{i,j},e^{+}_i\right\}}=x_{\left\{v_{i,j},e^{-}_i\right\}}=\frac{1}{\lambda -2}x_{e_i}.
\end{equation}
Let $\mathbf{x}_0$ denote the restriction of $\mathbf{x}$ to $E(K_n)$. Then $\mathbf{x}_0\neq \mathbf{0}$, for otherwise $x_{e_i}=0$ for all $1\le i\le 5$, and hence $\mathbf{x}=\mathbf{0}$ by \eqref{eq-3}, a contradiction.

 Combining \eqref{eq-3} with the eigenvalue equation,  for all $1\le i\le 5$, we have
\begin{equation}\label{eq-4}
	\begin{cases}
		\frac{\lambda a_i }{\lambda -2}x_{e_{i}} +\left(\mathcal{C}^\top(K_n) \mathcal{C}(K_n)\mathbf{x}_0\right)_{e_i}=\lambda x_{e_{i}}, &  1\le i\le 5,\\[2mm]
		\left(\mathcal{C}^\top(K_n) \mathcal{C}(K_n)\mathbf{x}_0\right)_{e}=\lambda x_{e}, & e\in E(K_n)\setminus \{e_i\mid  1\le i\le 5\}.\end{cases}
\end{equation}
Noticing that $\mathcal{B}(K_n)\mathcal{B}^{\top}(K_n)+\mathcal{C}^{\top} (K_n)\mathcal{C}(K_n)=\mathcal{H}(K_n)=nI$, it follows from \eqref{eq-4} that
\begin{equation}\label{eq-5}
	\begin{cases}
		(\lambda-n)x_{e_{i}}+\left(\mathcal{B}(K_n)\mathcal{B}^{\top}(K_n)\mathbf{x}_0\right)_{e_i}=	\frac{\lambda a_i }{\lambda -2}x_{e_{i}}, &  1\le i\le 5,\\[2mm]
	(\lambda-n)x_{e}+\left(\mathcal{B}(K_n)\mathcal{B}^{\top}(K_n)\mathbf{x}_0\right)_{e}=0, & e\in E(K_n)\setminus \{e_i\mid  1\le i\le 5\}.\end{cases}
\end{equation}
Set $\mu=\lambda-n>0$ and $S=\{e_1,e_2,e_3,e_4,e_5\}$. Rearranging the rows and columns so that the edges in $S$ correspond to the first five indices, \eqref{eq-5} can be rewritten as
\begin{equation}\label{eq-6}
	\bigl(\mu I+\mathcal{B}(K_n)\mathcal{B}^{\top}(K_n)\bigr)\mathbf{x}_0=\frac{\mu+n}{\mu+n-2}A\mathbf{x}_0,
\end{equation}
where
$
A=\operatorname{diag}(a_1,a_2,a_3,a_4,a_5,0,\ldots,0).
$
Since $\mu>0$ and $\mathcal{B}(K_n)\mathcal{B}^{\top}(K_n)$ is positive semi-definite, the matrix $\mu I+\mathcal{B}(K_n)\mathcal{B}^{\top}(K_n)$ is invertible. Hence, by \eqref{eq-6}, we get
\begin{equation}\label{eq-7}
	\mathbf{x}_0
	=
	\frac{\mu+n}{\mu+n-2}
	\bigl(\mu I+\mathcal{B}(K_n)\mathcal{B}^{\top}(K_n)\bigr)^{-1}
	A\mathbf{x}_0.
\end{equation}
It follows from \eqref{eq-02} and $\mathbf{s}\left(\Delta_0(K_n)\right)=\left\{n^{(n-1)},0\right\}$ that the only nonzero eigenvalue of $\mathcal{B}(K_n)\mathcal{B}^{\top}(K_n)$ is $n$.  Let $P$ and $Q$ denote the orthogonal projections, with respect to the standard orthonormal basis, onto the eigenspaces corresponding to the eigenvalues $n$ and $0$, respectively. Then $P+Q=I$, and the spectral decomposition of $\mathcal{B}(K_n)\mathcal{B}^{\top}(K_n)$ is 
\[\mathcal{B}(K_n)\mathcal{B}^{\top}(K_n)=nP+0Q=nP.\]
Consequently, we have $\mu I+\mathcal{B}(K_n)\mathcal{B}^{\top}(K_n)=\mu I+nP=(\mu+n)P+\mu Q$,
and hence
\[\begin{array}{ll}
		&\bigl(\mu I+\mathcal{B}(K_n)\mathcal{B}^{\top}(K_n))\bigr)^{-1}=\frac{1}{\mu+n}P+\frac{1}{\mu}Q\\[2mm]
        =&\frac{1}{\mu}I-\frac{n}{\mu(\mu+n)}P=\frac{1}{\mu}I-\frac{1}{\mu(\mu+n)}\mathcal{B}(K_n)\mathcal{B}^{\top}(K_n).
\end{array}\]
Substituting this into \eqref{eq-7}, we obtain
\begin{equation}\label{eq-a-3}
    \mathbf{x}_0=\frac{\mu+n}{\mu+n-2}\left(\frac{1}{\mu}I-\frac{1}{\mu(\mu+n)}\mathcal{B}(K_n)\mathcal{B}^{\top}(K_n)\right)A\mathbf{x}_0.
\end{equation}
Let $\mathbf{x}_S=\mathbf{x}_0|_S$ denote the restriction of $\mathbf{x}_0$ to the coordinates indexed by $S$. Denote by $L_S$ and $A_S$ the principal submatrices of $\mathcal{B}(K_n)\mathcal{B}^{\top}(K_n)$ and $A$, respectively, corresponding to the indices in $S$. From \eqref{eq-a-3}, we get
\[\mu\mathbf{x}_S=\left(\frac{\mu+n}{\mu+n-2}A_S-\frac{1}{\mu+n-2}L_SA_S\right)\mathbf{x}_S.\]
Multiplying both sides on the left by $A_S^{1/2}$, we further obtain
\begin{equation}\label{eq-8}
	\mu A^{\frac{1}{2}}_S\mathbf{x}_S=\left(\frac{\mu+n}{\mu+n-2}A_S-\frac{1}{\mu+n-2}A^{\frac{1}{2}}_SL_SA^{\frac{1}{2}}_S\right)A^{\frac{1}{2}}_S\mathbf{x}_S.
\end{equation}

Let $M(x) =\frac{x+n}{x+n-2}A_S-\frac{1}{x+n-2}A_S^{1/2}L_SA_S^{1/2}$.  For any $x>0$, let $g_1(x)\ge g_2(x)\ge g_3(x)\ge g_4(x)\ge g_5(x)$ denote the eigenvalues of $M(x)$ arranged in non-increasing order. For $M(x)$ with $x>0$, on the one hand,
by noticing that
\[\begin{aligned}
	M(x)&=\frac{x+n}{x+n-2}A_S-\frac{1}{x+n-2}A_S^{1/2}L_SA_S^{1/2} \\&=\frac{x}{x+n-2}A_S+\frac{1}{x+n-2}A_S^{1/2}(nI-L_S)A_S^{1/2},
\end{aligned}\]
we conclude that $M(x)$ is positive definite since $nI-L_S$ is positive semidefinite. Thus, $g_i(x)>0$ for all $1\le i\le 5$. On the other hand, by noticing that
\begin{equation}\label{eq-10}
	\begin{aligned}
		M(x)&=\frac{x+n}{x+n-2}A_S-\frac{1}{x+n-2}A_S^{1/2}L_SA_S^{1/2} \\&=A_S+\frac{1}{x+n-2}\left(2A_S-A_S^{1/2}L_SA_S^{1/2}\right)\\&=A_S+\frac{1}{x+n-2}A_S^{1/2}\left(2I-L_S\right)A_S^{1/2},
	\end{aligned}
\end{equation}
we get $\sum_{i=1}^{5}g_i(x)=\operatorname{Tr}(M(x))=\operatorname{Tr}(A_S)=\sum_{i=1}^{5}a_i$ since the diagonal entries of $2I-L_S$ are $0$ and $A_S=\operatorname{diag}(a_1,\ldots,a_5)$ is diagonal. Therefore, for any $x>0$, we have
\begin{equation}\label{eq-9}
	g_1(x)+g_2(x)\le \sum_{i=1}^{5}a_i.
\end{equation}

Equation \eqref{eq-8} indicates that $\mu$ is an eigenvalue of $M(\mu)$. Hence, there exists $1\le t\le 5$ such that $\mu=g_t(\mu)$. For $i=1,2$, define $\mu_i=\lambda_i\left(\operatorname{curl}^{*}\operatorname{curl}\right)-n$. Therefore, there exists $1\le t_i\le 5$ such that $\mu_i=g_{t_i}(\mu_i)$ for each $i\in\{1,2\}$,  i.e., each $\mu_i$ can be interpreted as an intersection point of the graphs $y=x$ and $y=g_{t_i}(x)$. 

\begin{claim}\label{cl-1}
    $g_1(x)$ is non-increasing on $(0,\infty)$.
\end{claim}
\begin{proof}
For each $x>0$, let $\mathbf y_x$ be a unit eigenvector of $M(x)$ corresponding to its largest eigenvalue $g_1(x)$. For any $\mathbf z \in \mathbb{R}^5$ indexed by $S$, from \eqref{eq-10}, we have \[\begin{aligned}
	\mathbf z^\top M(x)\mathbf z&= \mathbf z^\top A_S \mathbf z+ \frac{2}{x+n-2} \sum_{i<j} \sqrt{a_i a_j}\, \bigl(- (L_S)_{e_i,e_j}\bigr)\, \mathbf z_{e_i} \mathbf z_{e_j}\\&=\sum_{i=1}^{5}a_i\mathbf z_{e_i}^2+ \frac{2}{x+n-2} h(\mathbf{z}),
\end{aligned}\]
where $h(\mathbf z) =\sum_{i<j} \sqrt{a_i a_j}\, \bigl(- (L_S)_{e_i,e_j}\bigr)\, \mathbf z_{e_i} \mathbf z_{e_j}$. 
To establish that $g_1(x)$ is non-increasing on $(0,\infty)$, it suffices to show that $h(\mathbf y_x) \ge 0$ for all $x>0$, where $\mathbf y_x$ is a unit eigenvector of $M(x)$ corresponding to $g_1(x)$. Indeed, assuming this, for any $x_2 \ge x_1 > 0$,
\[
\begin{aligned}
	g_1(x_1) 
	&= \max_{\|\mathbf z\|=1} \mathbf z^\top M(x_1) \mathbf z 
	\ge \mathbf y_{x_2}^\top M(x_1) \mathbf y_{x_2} \\
	&= \mathbf y_{x_2}^\top A_S \mathbf y_{x_2} + \frac{2}{x_1+n-2} h(\mathbf y_{x_2}) 
	\ge \mathbf y_{x_2}^\top A_S \mathbf y_{x_2} + \frac{2}{x_2+n-2} h(\mathbf y_{x_2}) \\
	&= \mathbf y_{x_2}^\top M(x_2) \mathbf y_{x_2} = g_1(x_2),
\end{aligned}
\]
which implies that $g_1(x)$ is non-increasing on $(0,\infty)$.

It therefore remains to verify that $h(\mathbf y_x) \ge 0$ for all $x>0$. Since the verification is technical and tedious, we relegate it to Appendix \ref{app-2}.
\end{proof}

If $\mu_1=\mu_2=\mu$, then they are both eigenvalues of $M(\mu)$. Therefore, $\mu_1+\mu_2= g_{t_1}(\mu)+g_{t_2}(\mu)\le g_1(\mu)+g_2(\mu)\le\sum_{i=1}^5a_i$, and thus \eqref{eq-m-2} holds. If $\mu_1>\mu_2$, then from Claim \ref{cl-1}, we have
\[\mu_1+\mu_2=g_{t_1}(\mu_1)+g_{t_2}(\mu_2)\le g_{1}(\mu_1)+g_{t_2}(\mu_2)\le g_1(\mu_2)+g_{t_2}(\mu_2)\le \sum_{i=1}^5a_i,\]
and thus \eqref{eq-m-2} holds.

This completes the proof.

\end{proof}

\subsection{Proof of Corollary \ref{thm-main-2}}

From \eqref{eq-l-a-1}, we obtain that $\lambda_i(\Delta_1(G))\in \mathbf{s}(\Delta_0(G))\cup \mathbf{s}(\operatorname{curl}^*(G)\operatorname{curl}(G))$ for $i\in\{1,2\}$. If both $\lambda_1(\Delta_1(G))$ and $\lambda_2(\Delta_1(G))$ belong to $\mathbf{s}\left(\operatorname{curl}^*(G)\operatorname{curl}(G)\right)$, then 
\[\lambda_i(\Delta_1(G))=\lambda_i(\operatorname{curl}^*(G)\operatorname{curl}(G))\]
for $i\in\{1,2\}$.
From Theorem \ref{thm-main-1}, we have
\[\lambda_1(\Delta_1(G))+\lambda_2(\Delta_1(G)) \le D_1^\top(G)+D_2^\top(G).
\] 
If neither $\lambda_1(\Delta_1)$ nor $\lambda_2(\Delta_1)$ belongs to $\mathbf{s}\left(\operatorname{curl}^*(G)\operatorname{curl}(G)\right)$, then $\lambda_i(\Delta_1(G))\in \mathbf{s}(\Delta_0(G))$ for $i\in\{1,2\}$, and thereby
\[\lambda_i(\Delta_1(G))=\lambda_i(\Delta_0(G))\]
for $i\in\{1,2\}$.  By a result of Bai~\cite{B2011}, we have
\[\lambda_1(\Delta_1(G))+\lambda_2(\Delta_1(G)) \le d_1^\top(G)+d_2^\top(G).
\] 
Otherwise, exactly one of $\lambda_1(\Delta_1(G))$ and $\lambda_2(\Delta_1(G))$ belongs to $\mathbf{s}\left(\operatorname{curl}^*(G)\operatorname{curl}(G)\right)$,
while the other belongs to $\mathbf{s}\!\left(\Delta_0(G)\right)$.
Then, either 
\[\lambda_1(\Delta_1(G))=\lambda_1(\Delta_0(G)),\quad\lambda_2(\Delta_1(G))=\lambda_1(\operatorname{curl}^*(G)\operatorname{curl}(G)),\]
or
\[\lambda_1(\Delta_1(G))=\lambda_1(\operatorname{curl}^*(G)\operatorname{curl}(G)),\quad\lambda_2(\Delta_1(G))=\lambda_1(\Delta_0(G)).\]
Therefore, in either case, by Bai's inequality~\cite{B2011} and \eqref{eq-dr-1}, we have
\[\lambda_1(\Delta_1(G))+\lambda_2(\Delta_1(G)) \le d_1^\top(G)+D_1^\top(G).
\] 

This completes the proof.
\subsection{Proof of  Corollary \ref{cor-simplicial}}

The proof proceeds by proving the following two lemmas.
\begin{lemma}\label{lem-9}
Let $K$ be a $3$-family on vertex set $V(K)$ satisfying $d_2^\top(K)\le 2$. Then
\[
\lambda_1(\partial_{2}(K)\partial_{2}^*(K)) + \lambda_2(\partial_{2}(K)\partial_{2}^*(K)) \le d_1^\top(K) + d_2^\top(K),
\]
where $\partial_{2}(K)\partial_{2}^*(K)$ denotes the up-Laplacian of $K$.    
\end{lemma}
\begin{proof}
Let $G$ be the graph of $K$, that is, the graph with vertex set $V(G)=V(K)$ and edge set $E(G)$ being all the $2$-subsets of a set in $K$. If $\triangle=\{i,j,k\}\in T(G)$, then we claim $\{i,j,k\}\in K$. Otherwise, there exists $x,y,z\in V(K)$ such that $\{i,j,x\},\{j,k,y\},\{k,i,z\}\in K$. This leads to $d_i(K),d_j(K),d_k(K)\ge 2$, and thus $d_2^T(K)\ge 3$, a contradiction. Hence, $T(G)\subseteq K$. Conversely, for any set $\{i,j,k\}\in K$, we obtain that $\{i,j\},\{j,k\},\{k,i\}\in E(G)$, and thus $\{i,j,k\}\in T(G)$. Hence, $K\subseteq T(G)$. This yields $T(G)=K$. Therefore, we have $\partial_{2}(K)\partial_{2}^{*}(K)=\operatorname{curl}^*(G)\operatorname{curl}(G)$ and $D_i^\top(G)=d_i^\top(K)$ for $i\in\{1,2\}$.
It follows from Theorem \ref{thm-main-1} that 
\[\begin{array}{lll}
 &&\lambda_1(\partial_{2}(K)\partial_{2}^*(K)) + \lambda_2(\partial_{2}(K)\partial_{2}^*(K))\\[2mm]
 &=&\lambda_1\left(\operatorname{curl}^{*}(G)\operatorname{curl}(G)\right)+ \lambda_2\left(\operatorname{curl}^{*}(G)\operatorname{curl}(G)\right)\\[2mm]
 &\le& D_1^{\top}(G)+D_2^{\top}(G)=d_1^\top(K)+ d_2^\top(K),  
\end{array}\]
as desired.
\end{proof}

\begin{lemma}
Let $K$ be a $3$-family on vertex set $V(K)$ satisfying $d_2^\top(K)\ge 3$. Then
\[
\lambda_1(\partial_{2}(K)\partial_{2}^*(K)) + \lambda_2(\partial_{2}(K)\partial_{2}^*(K)) \le d_1^\top(K) + d_2^\top(K),
\]
where $\partial_{2}(K)\partial_{2}^*(K)$ denotes the up-Laplacian of $K$.    
\end{lemma}
\begin{proof}
According to Proposition 6.2 in \cite{DR2002}, we have
\begin{equation}\label{eq-t2-1}
    \lambda_1\!\left(\partial_{2}(K)\partial_{2}^*(K)\right)\le d_1^\top(K).
\end{equation}
Define \[V_1(K) = \left\{v \in V(K) \mid  \operatorname{deg}_K(v) = 1\right\}\text{~and~} V_2(K) = \left\{v \in V(K) \mid  \operatorname{deg}_K(v) \ge 2\right\}.\] 
Then $V(K)=V_1(K)\sqcup V_2(K)$, $|V(K)|=d_1^\top(K)$ and $|V_2(K)|=d_2^\top(K)$. We now distinguish two cases and discuss them separately.

\vspace{10pt}
\noindent\textbf{Case 1.} There is $\sigma\in K$ such that $\left|\sigma\cap V_1(K)\right|\ge 2$.
\vspace{10pt}

In this case, we proceed by induction on $|K|$. 
Since $d_2^\top(K)\ge 3$, we have $|K|\ge 3$. If $|K|=3$, one can verify that $K$ has the form like $K=\{\{i,j,k\},\{k,x,y\},\{x,y,z\}\}$, i.e., $K\cong T(G_4)$ (see Fig.~\ref{f-thm-4}). Thus, the result follows immediately from Theorem~\ref{thm-main-1}. Next, we consider the case for $|K|\ge 4$. If $\lambda_2(\partial_{2}(K)\partial_{2}^*(K))\le 3$, then, by \eqref{eq-t2-1} and the assumption $d_2^\top(K)\ge 3$, we get the result immediately. So, we may assume $\lambda_2(\partial_{2}(K)\partial_{2}^*(K))\ge 4$. Choose a $3$-set $\sigma\in K$ with $\left|\sigma\cap V_1(K)\right|\ge 2$. Let $K'=K\setminus\{\sigma\}$. Since there is no set $\sigma'\in K'$ with $|\sigma\cap \sigma'|=2$, we get
\[\mathbf{s}\left(\partial_{2}(K)\partial_{2}^*(K)\right)\circ{=}\mathbf{s}\left(\partial_{2}(K')\partial_{2}^*(K')\right)\uplus\{3\}.\]
This yields $\lambda_i(\partial_2(K)\partial_2^*(K))=\lambda_i(\partial_2(K')\partial_2^*(K'))$. Moreover, if $\left|\sigma\cap V_1(K)\right|=3$, then $d_1^\top(K')=d_1^\top(K)-3$ and $d_2^\top(K')=d_2^\top(K)$. If $\left|\sigma\cap V_1(K)\right|=2$, then we have either $d_1^\top(K')=d_1^\top(K)-2$ and $d_2^\top(K')=d_2^\top(K)$, or $d_1^\top(K')=d_1^\top(K)-1$ and $d_2^\top(K')=d_2^\top(K)-1$. Thus, we always have 
\[d_1^\top(K')+d_2^\top(K')\le d_1^\top(K)+d_2^\top(K)-2.\]
Note that $|K'|=|K|-1$. By the induction hypothesis, we get
\[\begin{array}{lll}
     &&\lambda_1\left(\partial_{2}(K)\partial_{2}^*(K)\right)+\lambda_2\left(\partial_{2}(K)\partial_{2}^*(K)\right) \\[2mm]
     &=&\lambda_1\left(\partial_{2}(K')\partial_{2}^*(K')\right)+\lambda_2\left(\partial_{2}(K')\partial_{2}^*(K')\right)\\[2mm]
     &\le &d_1^\top(K')+d_2^\top(K')\le d_1^\top(K)+d_2^\top(K)-2,
\end{array}
	\]
as desired.

\vspace{10pt}
\noindent\textbf{Case 2.} For any $\sigma\in K$, $\left|\sigma\cap V_1(K)\right|\le 1$.
\vspace{10pt}

If $V_1(K)=\emptyset$, then $d_1^\top(K)=d_2^\top(K)$. It follows from \eqref{eq-t2-1} that
\[
\lambda_1\!\left(\partial_{2}(K)\partial_{2}^*(K)\right)+\lambda_2\!\left(\partial_{2}(K)\partial_{2}^*(K)\right) \le 2\lambda_1\!\left(\partial_{2}(K)\partial_{2}^*(K)\right) \le 2d_1^\top(K)= d_1^\top(K)+d_2^\top(K),
\]
as desired. If $V_1(K)\ne \emptyset$, let $G$ be the graph of $K$ as defined in the proof of Lemma \ref{lem-9}. Therefore, $K\subseteq T(G)$ as proved in the proof of Lemma \ref{lem-9}. This implies that the matrix representation of $\partial_2^*(K)\partial_2(K)$ is a principal submatrix of that of $\partial_2^*(T(G))\partial_2(T(G))$. Hence, by the interlacing theorem and the fact that $\mathbf{s}(\partial_{2}^*\partial_{2}) \stackrel{\circ}{=} \mathbf{s}(\partial_{2}\partial_{2}^*)$, we have
\begin{equation}\label{eq-thm-z-3}
\begin{array}{ll}
    &\lambda_1(\partial_{2}(K)\partial_{2}^*(K)) + \lambda_2(\partial_{2}(K)\partial_{2}^*(K))\\[2mm]
    \le& \lambda_1(\partial_{2}(T(G))\partial_{2}^*(T(G))) + \lambda_2(\partial_{2}(T(G))\partial_{2}^*(T(G))).
    \end{array}
\end{equation}
Recall the definitions of $V_1(G)$ and $V_2(G)$ from the second paragraph of Section \ref{se-3}. For the graph $G$, we have $V(G)=V(K)$,
$V_1(G)=V_1(K)$ and $V_2(G)=V_2(K)$, and thus 
\begin{equation}\label{eq-thm-z-4}
    d_1^\top(K)=|V(K)|=|V(G)|=D_1^{\top}(G)\text{  and  } d_2^\top(K)=|V_2(K)|=|V_2(G)|=D_2^{\top}(G).
\end{equation}
Since $\partial_{2}(T(G))\partial_{2}^{*}(T(G))=\operatorname{curl}^*(G)\operatorname{curl}(G)$, it follows from Theorem \ref{thm-main-1} and \eqref{eq-thm-z-4} that 
\begin{equation}\label{eq-thm-z-5}
\begin{array}{ll}
     & \lambda_1(\partial_{2}(T(G))\partial_{2}^*(T(G))) + \lambda_2(\partial_{2}(T(G))\partial_{2}^*(T(G))) \\[2mm]
     \le &  D_1^\top(G) + D_2^\top(G)=d_1^\top(K)+d_2^\top(K).
\end{array}
\end{equation}
Therefore, combining \eqref{eq-thm-z-3} with \eqref{eq-thm-z-5}, we have
\[
\lambda_1\!\left(\partial_{2}(K)\partial_{2}^*(K)\right)+\lambda_2\!\left(\partial_{2}(K)\partial_{2}^*(K)\right) \le d_1^\top(K)+d_2^\top(K).\]

This completes the proof.
\end{proof}

\section*{Declaration of Interest Statement}
The authors declare that they have no conflicts of interest to this work.

\section*{Acknowledgments}
Lu Lu is supported by National Natural Science Foundation of China (No. 12371362). Jianfeng Wang is supported by National Natural Science Foundation of China (No. 12371353).

\addcontentsline{toc}{section}{References}  
\bibliographystyle{plain}

{\small\begin{thebibliography}{99}
\bibitem{B2011} H. Bai, The Grone-Merris conjecture, Trans. Amer. Math. Soc. 363 (2011), no. 8, 4463–4474.

\bibitem{CRS2010} D. Cvetkovi\'c, P. Rowlinson, S. Simi\'c, An introduction to the theory of graph spectra. London Mathematical Society Student Texts. Cambridge University Press, Cambridge, 2010.

\bibitem{DLM2005} M. Desbrun, M. Leok, and J. E. Marsden, Discrete Poincar\'e lemma, Appl. Numer. Math., 53 (2005), pp. 231--248.

\bibitem{DHK2011} T. K. Dey, A. N. Hirani, and B. Krishnamoorthy, Optimal homologous cycles, total unimodularity, and linear programming, SIAM J. Comput., 40 (2011), pp. 1026--1044.

\bibitem{DR2002} A. M. Duval, V. Reiner, Shifted simplicial complexes are Laplacian integral, Trans. Amer. Math. Soc. 354 (11)(2002) 4313–4344.
\bibitem{FWW2024} Y. Z. Fan, H. F. Wu, Y. Wang, The largest Laplacian eigenvalue and the balancedness of simplicial complexes, J. Algebraic Comb. 61 (53) (2025).

\bibitem{GR2001} C. Godsil, G. Royle, Algebraic Graph Theory, Springer-Verlag, New York, 2001.

\bibitem{GM1994} R. Grone, R. Merris, The Laplacian spectrum of a graph. II. SIAM J. Discrete Math., 7(2)(1994) 221--229.

\bibitem{JLYY2011} X. Jiang, L. H. Lim, Y. Yao, Y. Ye, Statistical ranking and combinatorial Hodge theory, Math. Program., 127 (2011), pp. 203--244.

\bibitem{LHL2020}  L.-H. Lim, Hodge Laplacians on graphs, SIAM Rev. 62 (2020) 685--715.

\bibitem{MOA2011} A. W. Marshall, I. Olkin, and B. C. Arnold. Inequalities: theory of majorization and its applications. Springer Series in Statistics. Springer, New York, second edition, 2011.



\bibitem{FWW1983} F. W. Warner, Foundations of Differentiable Manifolds and Lie Groups, Grad. Texts in Math. 94, Springer-Verlag, New York, 1983. 

	\end{thebibliography}}
\appendix
\section{The eigenvalues of graphs in Fig.\ref{f-thm-1}}\label{app-1}
For $G\in \left\{t(G)K_3,F_s\cup (t(G)-s)K_3,G_1\right\}$, no two triangles of $G$ share an edge. Therefore, regardless of the choice of orientation, the matrix representation of $\operatorname{curl}\operatorname{curl}^*$ is diagonal and equals
\[
\operatorname{diag}(\underbrace{3,\ldots,3}_{t(G)}).
\]
Since $\mathbf{s}(\operatorname{curl}^*\operatorname{curl}) \stackrel{\circ}{=} \mathbf{s}(\operatorname{curl}\operatorname{curl}^*)$ and $t(G)\ge 2$, we have \begin{equation}\label{eq-f-1}
\lambda_1(\operatorname{curl}^*~\operatorname{curl})+\lambda_2(\operatorname{curl}^*~\operatorname{curl})=6.
\end{equation}
Moreover, if $G=t(G)K_3$ and $t(G)\ge 2$, then \[D_1^{\top}(G)+D_2^{\top}(G)=3t(G)\ge 6.\] 
If $G=F_s\cup (t(G)-s)K_3$, then $t(G)\ge s\ge 2$,  and thus
\[D_1^{\top}(G)+D_2^{\top}(G)=3t(G)-s+1+1\ge 2t(G)+2\ge 6.\]
If $G=G_1$, then $t(G)\ge \deg_G^{(2)}(u)+\deg_G^{(2)}(v)\ge 4$, and thus
\[D_1^{\top}(G)+D_2^{\top}(G)=3t(G)-\deg_G^{(2)}(u)-\deg_G^{(2)}(v)+2\ge 2t(G)+2\ge 6.\]
Therefore, in all cases,
\[
D_1^{\top}(G)+D_2^{\top}(G)\ge 6.
\]
Combining this with \eqref{eq-f-1}, we obtain
\[
\lambda_1(\operatorname{curl}^*\operatorname{curl})
+
\lambda_2(\operatorname{curl}^*\operatorname{curl})
\le
D_1^{\top}(G)+D_2^{\top}(G).
\]

For $G=G_2$, let
\[e:=\{u,v\},\qquad d_u:=\deg_G^{(2)}(u),\qquad d_v:=\deg_G^{(2)}(v),\qquad d_e:=\deg_G(e).\]
By inclusion--exclusion, we get $d_u+d_v-d_e\le t(G)$.
Hence,
\begin{equation}\label{eq-f-2}
    D_1^{\top}(G)+D_2^{\top}(G)=3t(G)-(d_u-1)-(d_v-1)+2=3t(G)-d_u-d_v+4\ge 2t(G)-d_e+4.
\end{equation}
Define \[T_1(G)\mid =\{\triangle\in T(G) \mid e\in\triangle\}\text{ and }T_2(G)\mid =T(G)\setminus\{\triangle\in T(G) \mid e\in\triangle\}.\]
Since any two triangles in $T_1(G)$ share the edge $e$, while no two triangles in $T_2(G)$ share an edge and no triangle in $T_1(G)$ shares an edge with a triangle in $T_2(G)$, the matrix representation $\mathcal{C}\mathcal{C}^\top$ of $\operatorname{curl}\operatorname{curl}^*$ is block diagonal. More precisely, after a suitable orientation and ordering of the triangles, we get
\[
\mathcal{C}\mathcal{C}^\top
=
\begin{pmatrix}
2I_{d_e}+J_{d_e} & 0\\
0 & \operatorname{diag}(\underbrace{3,\ldots,3}_{\,t(G)-d_e})
\end{pmatrix},
\]
where $J_{d_e}$ denotes the $d_e\times d_e$ all-ones matrix. Therefore, the non-zero spectrum of $\operatorname{curl}^*\operatorname{curl}$ is given by
\begin{equation}\label{eq-f-3}
\mathbf{s}\left(\operatorname{curl}^*\operatorname{curl}\right)\stackrel{\circ}{=} \mathbf{s}\left(\operatorname{curl}\operatorname{curl}^*\right) \stackrel{\circ}{=} \{d_e+2, 3^{\left(t(G)-d_e\right)}, 2^{(d_e-1)}\}.
\end{equation}
If $d_e\le t(G)-1$, then $d_e\le 2t(G)-d_e-2$. Combining this with \eqref{eq-f-2}, we have \[
\lambda_1(\operatorname{curl}^*\operatorname{curl})+
\lambda_2(\operatorname{curl}^*\operatorname{curl})=d_e+5\le  2t(G)-d_e+3< 2t(G)-d_e+4\le D_1^{\top}(G)+D_2^{\top}(G).
\]
Otherwise, $d_e= t(G)$, and from \eqref{eq-f-2} and \eqref{eq-f-3}, we have
\[
\lambda_1(\operatorname{curl}^*\operatorname{curl})+
\lambda_2(\operatorname{curl}^*\operatorname{curl})=d_e+4\le D_1^{\top}(G)+D_2^{\top}(G),
\] as desired.
\section{Non-negativity estimate}\label{app-2}
Recall that $S=\{e_1,e_2,e_3,e_4,e_5\}$, and that $L_S$ denotes the principal submatrix of $\mathcal{B}(K_n)\mathcal{B}(K_n)^\top$ induced by $S$. Moreover, recall that
\[\mathbf z^\top M(x)\mathbf z=\sum_{i=1}^{5} a_i \mathbf z_{e_i}^2+\frac{2}{x+n-2} h(\mathbf z),\]
where $h(\mathbf z):=\sum_{i<j}\sqrt{a_i a_j}\, \bigl(- (L_S)_{e_i,e_j}\bigr)\, \mathbf z_{e_i}\mathbf z_{e_j}.$
Here $\mathbf z \in \mathbb{R}^5$ is indexed by $S$, and $x,a_i>0$ for $1\le i\le 5$. For each $x>0$,  let $\mathbf y_x$ denote a unit eigenvector of $M(x)$ corresponding to its largest eigenvalue $g_1(x)$. In this appendix, we prove that \[h(\mathbf y_x) \ge 0 \quad \text{for all } x>0.\]

{\begin{center}
\setlength\LTleft{\fill}
\setlength\LTright{\fill}
\begin{longtable}{|c|c|c|c|}
\caption{\small All the 26 configurations with suitable orientations}\label{tab-f} \\ 
\hline
\endfirsthead  
\multicolumn{4}{c}{\textit{Continued from previous page}} \\
\hline
\endhead  
\hline
\multicolumn{4}{r}{\textit{Continued on next page}} \\
\endfoot  
\hline
\endlastfoot 

\includegraphics[width=0.18\textwidth]{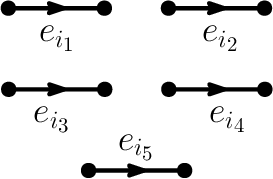} & 
\includegraphics[width=0.18\textwidth]{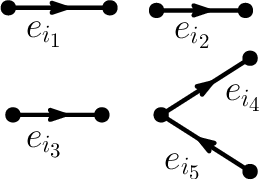} & 
\includegraphics[width=0.18\textwidth]{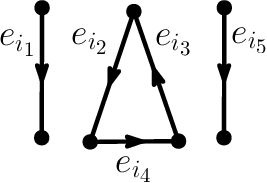} & 
\includegraphics[width=0.18\textwidth]{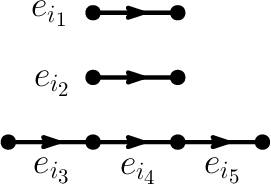} \\
(1) & (2) & (3) & (4) \\
\hline
\includegraphics[width=0.18\textwidth]{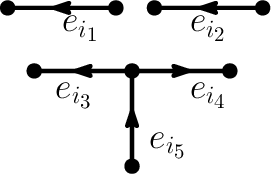} & 
\includegraphics[width=0.18\textwidth]{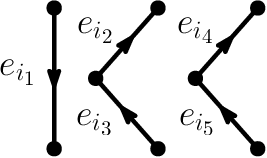} & 
\includegraphics[width=0.18\textwidth]{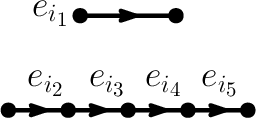} & 
\includegraphics[width=0.18\textwidth]{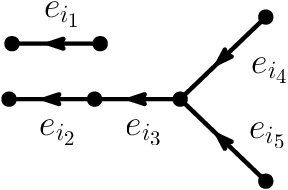} \\
(5) & (6) & (7) & (8) \\
\hline
\includegraphics[width=0.18\textwidth]{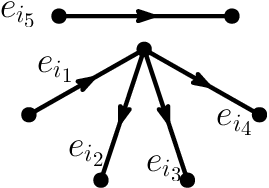} & 
\includegraphics[width=0.18\textwidth]{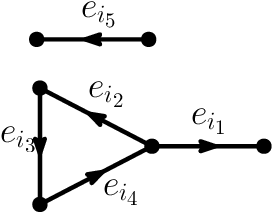} & 
\includegraphics[width=0.18\textwidth]{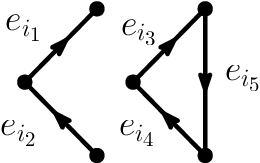} & 
\includegraphics[width=0.18\textwidth]{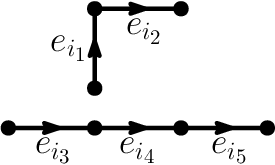} \\
(9) & (10) & (11) & (12) \\
\hline
\includegraphics[width=0.18\textwidth]{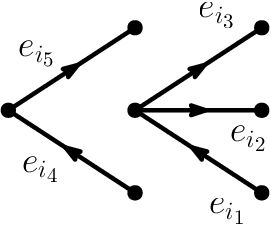} & 
\includegraphics[width=0.18\textwidth]{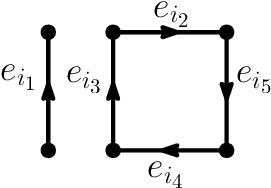} & 
\includegraphics[width=0.18\textwidth]{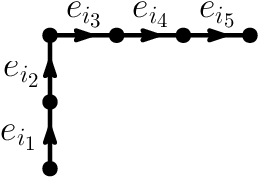} & 
\includegraphics[width=0.18\textwidth]{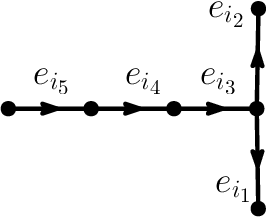} \\
(13) & (14) & (15) & (16) \\
\hline
\includegraphics[width=0.18\textwidth]{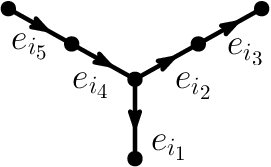} & 
\includegraphics[width=0.18\textwidth]{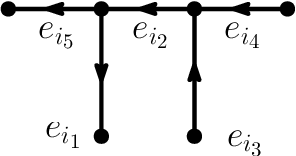} & 
\includegraphics[width=0.18\textwidth]{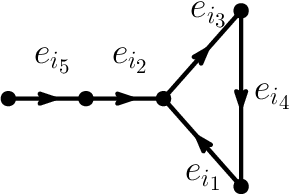} & 
\includegraphics[width=0.18\textwidth]{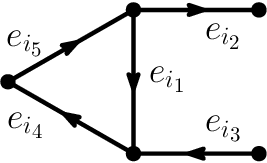} \\
(17) & (18) & (19) & (20) \\
\hline
\includegraphics[width=0.18\textwidth]{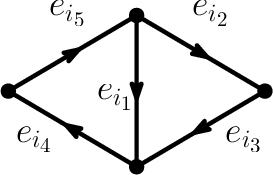} & 
\includegraphics[width=0.18\textwidth]{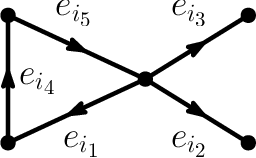} & 
\includegraphics[width=0.18\textwidth]{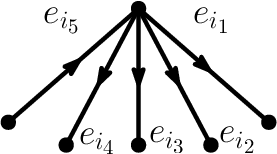} & 
\includegraphics[width=0.18\textwidth]{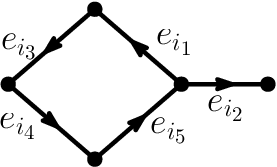} \\
(21) & (22) & (23) & (24) \\
\hline
\includegraphics[width=0.18\textwidth]{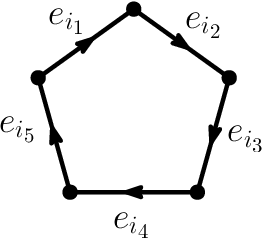} & 
\includegraphics[width=0.18\textwidth]{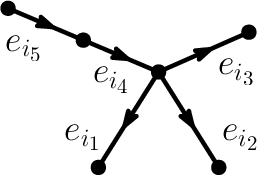} & 
 & 
 \\
(25) & (26) &  & \\
\hline
\end{longtable}\end{center}}

We now analyze the structure of $G[S]$. Up to graph isomorphism, there are exactly $26$ distinct configurations. For each configuration, we proceed as follows.
\begin{itemize}[leftmargin=3.5em, labelwidth=2.5em, labelsep=0.5em]
    \item[Step 1.] Choose a suitable orientation for the configuration.
    \item[Step 2.] Assume for contradiction that $h(\mathbf{y}_x)<0$.
    \item[Step 3.] Based on the chosen orientation, select a vector $\mathbf{z}_x\in\mathbb{R}^5$ satisfying 
    \[\|\mathbf{z}_x\|=\|\mathbf{y}_x\|\quad\text{and}\quad 
    \mathbf{z}_x^{\top}M(x)\mathbf{z}_x>\mathbf{y}_x^{\top}M(x)\mathbf{y}_x=g_1(x).\]
    \item[Step 4.] This contradicts the fact that $g_1(x)$ is the largest eigenvalue of $M(x)$.
\end{itemize}
We provide the details below.

\begin{itemize}
\item For the configurations (1), (2), (3), (4), (6), (7), (11), (12), (14), (15), (25).

Since $\Delta(G[S])\le 2$, there exists an orientation satisfying {\rm (i)} and {\rm (ii)} as given in the proof of Lemma \ref{lem-main-1} such that every pair of adjacent edges in $S$ is oriented consecutively through their common endpoint (see Table~\ref{tab-f}). Under this orientation, $(L_S)_{e_i,e_j}\in \{0,-1\}$ for all $i\ne j$, and thus $M(x)$ is a nonnegative symmetric matrix for every $x\ge 0$. Suppose, to the contrary, that $h(\mathbf y_x)<0$. Let $\mathbf z_x=|\mathbf{y}_x|$ be the vector in $\mathbb{R}^5$ whose components are the absolute values of the corresponding entries of $\mathbf{y}_x$. Since
\[
h(\mathbf z_x)=\sum_{i<j}\sqrt{a_i a_j}\,\bigl(-(L_S)_{e_i,e_j}\bigr)(\mathbf z_x)_{e_i}(\mathbf z_x)_{e_j}\ge 0 > h(\mathbf y_x),
\]
and 
\[
\|\mathbf z_x\|=\|\mathbf y_x\|=1,\qquad 
\sum_{i=1}^{5} a_i (\mathbf z_x)_{e_i}^2=\sum_{i=1}^{5} a_i (\mathbf y_x)_{e_i}^2,
\]
we obtain
\[
\begin{aligned}
\mathbf z_x^{\top}M(x)\mathbf z_x
&= \sum_{i=1}^{5} a_i (\mathbf z_x)_{e_i}^2+\frac{2}{x+n-2}h(\mathbf z_x)\\
&> \sum_{i=1}^{5} a_i (\mathbf y_x)_{e_i}^2+\frac{2}{x+n-2}h(\mathbf y_x)\\
&= \mathbf y_x^{\top}M(x)\mathbf y_x= g_1(x).
\end{aligned}
\]

\item {For the configurations (5).}

	We fix an orientation satisfying {\rm (i)} and {\rm (ii)} such that the edges in $S$ are oriented as depicted in Table~\ref{tab-f}. Under this orientation, for any $x>0$,
\[
h(\mathbf y_x)
=
\sqrt{a_{i_3}a_{i_5}}(\mathbf y_x)_{e_{i_3}}(\mathbf y_x)_{e_{i_5}}
+\sqrt{a_{i_4}a_{i_5}}(\mathbf y_x)_{e_{i_4}}(\mathbf y_x)_{e_{i_5}}
-\sqrt{a_{i_3}a_{i_4}}(\mathbf y_x)_{e_{i_3}}(\mathbf y_x)_{e_{i_4}}.
\]
Suppose, to the contrary, that $h(\mathbf y_x)<0$. If
\[
\sqrt{a_{i_5}}\bigl|(\mathbf y_x)_{e_{i_5}}\bigr|=
\max_{j\in\{3,4,5\}}
\sqrt{a_{i_j}}\bigl|(\mathbf y_x)_{e_{i_j}}\bigr|,
\]
set $\mathbf z_x=|\mathbf y_x|$. Otherwise, let $t\in\{3,4\}$ such that \[
\sqrt{a_{i_t}}\bigl|(\mathbf y_x)_{e_{i_t}}\bigr|=
\max_{j\in\{3,4,5\}}
\sqrt{a_{i_j}}\bigl|(\mathbf y_x)_{e_{i_j}}\bigr|,
\]
and define $\mathbf z_x\in\mathbb R^5$ by
\[
(\mathbf z_x)_{e_{i_j}}=\begin{cases}
	-\bigl|(\mathbf y_x)_{e_{i_j}}\bigr|, & j\in \{t,5\},\\
	\bigl|(\mathbf y_x)_{e_{i_j}}\bigr|, &\text{ otherwise}.
\end{cases}
\]
In either case,
\[h(\mathbf z_x)\ge0>h(\mathbf y_x),\]
while
\[\|\mathbf z_x\|=\|\mathbf y_x\|=1,\qquad\sum_{i=1}^{5} a_i (\mathbf z_x)_{e_i}^2=\sum_{i=1}^{5} a_i (\mathbf y_x)_{e_i}^2.\]
Consequently,
\[\begin{aligned}\mathbf z_x^{\top}M(x)\mathbf z_x
	&=	\sum_{i=1}^{5} a_i (\mathbf z_x)_{e_i}^2+\frac{2}{x+n-2}h(\mathbf z_x)\\
	&>\sum_{i=1}^{5} a_i (\mathbf y_x)_{e_i}^2+\frac{2}{x+n-2}h(\mathbf y_x)\\
	&=\mathbf y_x^{\top}M(x)\mathbf y_x\\
	&=g_1(x).\end{aligned}\]

\item {For the configurations (8).}

	We fix an orientation satisfying {\rm (i)} and {\rm (ii)} such that the edges in $S$ are oriented as depicted in Table~\ref{tab-f}. Under this orientation, for any $x>0$,
{\[\begin{array}{lll}
h(\mathbf y_x)
&=&
\sqrt{a_{i_2}a_{i_3}}(\mathbf y_x)_{e_{i_2}}(\mathbf y_x)_{e_{i_3}}
+\sqrt{a_{i_3}a_{i_4}}(\mathbf y_x)_{e_{i_3}}(\mathbf y_x)_{e_{i_4}}\\[2mm]
&&+\sqrt{a_{i_3}a_{i_5}}(\mathbf y_x)_{e_{i_3}}(\mathbf y_x)_{e_{i_5}}
-\sqrt{a_{i_4}a_{i_5}}(\mathbf y_x)_{e_{i_4}}(\mathbf y_x)_{e_{i_5}}.
\end{array}\]}
Suppose, to the contrary, that $h(\mathbf y_x)<0$. If
\[\sqrt{a_{i_3}}\bigl|(\mathbf y_x)_{e_{i_3}}\bigr|=\max_{j\in\{3,4,5\}}
\sqrt{a_{i_j}}\bigl|(\mathbf y_x)_{e_{i_j}}\bigr|,\]
set $\mathbf z_x:=|\mathbf y_x|$. Otherwise, let $t\in\{4,5\}$ such that \[\sqrt{a_{i_t}}\bigl|(\mathbf y_x)_{e_{i_t}}\bigr|=\max_{j\in\{3,4,5\}}
\sqrt{a_{i_j}}\bigl|(\mathbf y_x)_{e_{i_j}}\bigr|,\]
and define $\mathbf z_x\in\mathbb R^5$ by
\[
(\mathbf z_x)_{e_{i_j}}=\begin{cases}	-|(\mathbf y_x)_{e_{i_j}}|, & j\in\{t,2,3\},\\
	|(\mathbf y_x)_{e_{i_j}}|, & \text{ otherwise}.
\end{cases}
\]
In either case,
\[h(\mathbf z_x)\ge0>h(\mathbf y_x),\]
while
\[\|\mathbf z_x\|=\|\mathbf y_x\|=1,\qquad\sum_{i=1}^{5} a_i (\mathbf z_x)_{e_i}^2=\sum_{i=1}^{5} a_i (\mathbf y_x)_{e_i}^2.\]
Consequently,
\[\begin{aligned}\mathbf z_x^{\top}M(x)\mathbf z_x
		&=	\sum_{i=1}^{5} a_i (\mathbf z_x)_{e_i}^2+\frac{2}{x+n-2}h(\mathbf z_x)\\
		&>\sum_{i=1}^{5} a_i (\mathbf y_x)_{e_i}^2+\frac{2}{x+n-2}h(\mathbf y_x)\\
		&=\mathbf y_x^{\top}M(x)\mathbf y_x\\
		&=g_1(x).\end{aligned}\]

\item{For the configurations (9).}

We fix an orientation satisfying {\rm (i)} and {\rm (ii)} such that the edges in $S$ are oriented as shown in Table~\ref{tab-f}. Under this orientation,
\[
\begin{aligned}
	h(\mathbf y_x)
	&=\sqrt{a_{i_1}a_{i_3}}\,(\mathbf y_x)_{e_{i_1}}(\mathbf y_x)_{e_{i_3}}+\sqrt{a_{i_1}a_{i_2}}\,(\mathbf y_x)_{e_{i_1}}(\mathbf y_x)_{e_{i_2}}
	+\sqrt{a_{i_1}a_{i_4}}\,(\mathbf y_x)_{e_{i_1}}(\mathbf y_x)_{e_{i_4}}	\\&\quad-\sqrt{a_{i_2}a_{i_3}}\,(\mathbf y_x)_{e_{i_2}}(\mathbf y_x)_{e_{i_3}}-\sqrt{a_{i_2}a_{i_4}}\,(\mathbf y_x)_{e_{i_2}}(\mathbf y_x)_{e_{i_4}}-\sqrt{a_{i_3}a_{i_4}}\,(\mathbf y_x)_{e_{i_3}}(\mathbf y_x)_{e_{i_4}}.
\end{aligned}
\]
Suppose, to the contrary, that
$h(\mathbf y_x)<0$. If \[\sqrt{a_{i_1}}\bigl|(\mathbf y_x)_{e_{i_1}}\bigr|=\max_{j\in\{1,2,3,4\}}\sqrt{a_{i_j}}\bigl|(\mathbf y_x)_{e_{i_j}}\bigr|,\] 
then set $\mathbf z_x\mid =|\mathbf y_x|$.
Otherwise,  let $t\in\{2,3,4\}$ such that
\[\sqrt{a_{i_t}}\bigl|(\mathbf y_x)_{e_{i_t}}\bigr|=\max_{j\in\{1,2,3,4\}}\sqrt{a_{i_j}}\bigl|(\mathbf y_x)_{e_{i_j}}\bigr|,\] and 
define $\mathbf z_x\in\mathbb R^5$ by
\[(\mathbf z_x)_{e_{i_j}}=\begin{cases}	-|(\mathbf y_x)_{e_{i_j}}|, & j\in \{1,t\},\\[1mm]	|(\mathbf y_x)_{e_{i_j}}|, & \text{otherwise}.
\end{cases}
\]
A direct verification shows that, in each of the above cases, $h(\mathbf z_x)\ge 0>h(\mathbf y_x)$.
Moreover, since $\|\mathbf z_x\|=\|\mathbf y_x\|=1$ and $\sum_{i=1}^{5} a_i (\mathbf z_x)_{e_i}^2=\sum_{i=1}^{5} a_i (\mathbf y_x)_{e_i}^2$, we \[\begin{aligned}\mathbf z_x^{\top}M(x)\mathbf z_x
		&=	\sum_{i=1}^{5} a_i (\mathbf z_x)_{e_i}^2+\frac{2}{x+n-2}h(\mathbf z_x)\\
		&>\sum_{i=1}^{5} a_i (\mathbf y_x)_{e_i}^2+\frac{2}{x+n-2}h(\mathbf y_x)\\
		&=\mathbf y_x^{\top}M(x)\mathbf y_x\\
		&=g_1(x).\end{aligned}\]

\item{For the configurations (10).}

We fix an orientation satisfying {\rm (i)} and {\rm (ii)} such that the edges in $S$ are oriented as depicted in Table~\ref{tab-f}. Under this orientation, for any $x>0$,
\[
\begin{aligned}
	h(\mathbf y_x) &= 
	\sqrt{a_{i_2}a_{i_3}}\, (\mathbf y_x)_{e_{i_2}} (\mathbf y_x)_{e_{i_3}}
	+ \sqrt{a_{i_3}a_{i_4}}\, (\mathbf y_x)_{e_{i_3}} (\mathbf y_x)_{e_{i_4}}+ \sqrt{a_{i_2}a_{i_4}}\, (\mathbf y_x)_{e_{i_2}} (\mathbf y_x)_{e_{i_4}}
	\\	&\quad + \sqrt{a_{i_1}a_{i_4}}\, (\mathbf y_x)_{e_{i_1}} (\mathbf y_x)_{e_{i_4}}- \sqrt{a_{i_1}a_{i_2}}\, (\mathbf y_x)_{e_{i_1}} (\mathbf y_x)_{e_{i_2}}.
\end{aligned}
\]
Suppose, to the contrary, that $h(\mathbf y_x)<0$. If \[\sqrt{a_{i_4}}\bigl|(\mathbf y_x)_{e_{i_4}}\bigr|\ge \sqrt{a_{i_2}}\bigl|(\mathbf y_x)_{e_{i_2}}\bigr|,\]
then set $\mathbf z_x:=|\mathbf y_x|$. Otherwise, define $\mathbf z_x\in\mathbb R^5$ by
\[(\mathbf z_x)_{e_{i_j}}=\begin{cases}-|(\mathbf y_x)_{e_{i_j}}|, & j=1\\|(\mathbf y_x)_{e_{i_j}}|, & \text{otherwise}.
\end{cases}
\]
In either case,
\[h(\mathbf z_x)\ge0>h(\mathbf y_x),\]
while
\[\|\mathbf z_x\|=\|\mathbf y_x\|=1,\qquad\sum_{i=1}^{5} a_i (\mathbf z_x)_{e_i}^2=\sum_{i=1}^{5} a_i (\mathbf y_x)_{e_i}^2.\]
Consequently,
\[\begin{aligned}\mathbf z_x^{\top}M(x)\mathbf z_x
		&=	\sum_{i=1}^{5} a_i (\mathbf z_x)_{e_i}^2+\frac{2}{x+n-2}h(\mathbf z_x)\\
		&>\sum_{i=1}^{5} a_i (\mathbf y_x)_{e_i}^2+\frac{2}{x+n-2}h(\mathbf y_x)\\
		&=\mathbf y_x^{\top}M(x)\mathbf y_x\\
		&=g_1(x).\end{aligned}\]

\item {For the configurations (13).}

We fix an orientation satisfying {\rm (i)} and {\rm (ii)} such that the edges in $S$ are oriented as depicted in Table~\ref{tab-f}. Under this orientation, for any $x>0$,
\[
\begin{aligned}
	h(\mathbf y_x) &=  \sqrt{a_{i_1}a_{i_2}}\, (\mathbf y_x)_{e_{i_1}} (\mathbf y_x)_{e_{i_2}}+ \sqrt{a_{i_1}a_{i_3}}\, (\mathbf y_x)_{e_{i_1}} (\mathbf y_x)_{e_{i_3}}+ \sqrt{a_{i_4}a_{i_5}}\, (\mathbf y_x)_{e_{i_4}} (\mathbf y_x)_{e_{i_5}}
	\\	&\quad -\sqrt{a_{i_2}a_{i_3}}\, (\mathbf y_x)_{e_{i_2}} (\mathbf y_x)_{e_{i_3}}. 
\end{aligned}\]
Suppose, to the contrary, that $h(\mathbf y_x)<0$. If
\[\sqrt{a_{i_1}}\bigl|(\mathbf y_x)_{e_{i_1}}\bigr|=\max_{j\in\{1,2,3\}}\sqrt{a_{i_j}}\bigl|(\mathbf y_x)_{e_{i_j}}\bigr|,\]
set $\mathbf z_x:=|\mathbf y_x|$. Otherwise, let $t\in \{2,3\}$ such that \[\sqrt{a_{i_t}}\bigl|(\mathbf y_x)_{e_{i_t}}\bigr|=\max_{j\in\{1,2,3\}}\sqrt{a_{i_j}}\bigl|(\mathbf y_x)_{e_{i_j}}\bigr|,\]
and define $\mathbf z_x\in\mathbb R^5$ by
\[(\mathbf z_x)_{e_{i_j}}=\begin{cases}
	-|(\mathbf y_x)_{e_{i_j}}|, & j=\{1,t\}\\
	|(\mathbf y_x)_{e_{i_j}}|, &\text{otherwise}.\end{cases}\]
In either case,
\[h(\mathbf z_x)\ge0>h(\mathbf y_x),\]
while
\[\|\mathbf z_x\|=\|\mathbf y_x\|=1,\qquad\sum_{i=1}^{5} a_i (\mathbf z_x)_{e_i}^2=\sum_{i=1}^{5} a_i (\mathbf y_x)_{e_i}^2.\]
Consequently,
\[\begin{aligned}\mathbf z_x^{\top}M(x)\mathbf z_x
		&=	\sum_{i=1}^{5} a_i (\mathbf z_x)_{e_i}^2+\frac{2}{x+n-2}h(\mathbf z_x)\\
		&>\sum_{i=1}^{5} a_i (\mathbf y_x)_{e_i}^2+\frac{2}{x+n-2}h(\mathbf y_x)\\
		&=\mathbf y_x^{\top}M(x)\mathbf y_x\\
		&=g_1(x).\end{aligned}\]

\item {For the configurations (16).}

We fix an orientation satisfying {\rm (i)} and {\rm (ii)} such that the edges in $S$ are oriented as depicted in Table~\ref{tab-f}. Under this orientation, for any $x>0$,
\[
\begin{aligned}
	h(\mathbf y_x) &=  \sqrt{a_{i_2}a_{i_3}}\, (\mathbf y_x)_{e_{i_2}} (\mathbf y_x)_{e_{i_3}}+ \sqrt{a_{i_1}a_{i_3}}\, (\mathbf y_x)_{e_{i_1}} (\mathbf y_x)_{e_{i_3}}+ \sqrt{a_{i_4}a_{i_5}}\, (\mathbf y_x)_{e_{i_4}} (\mathbf y_x)_{e_{i_5}}
	\\	&\quad + \sqrt{a_{i_3}a_{i_4}}\, (\mathbf y_x)_{e_{i_3}} (\mathbf y_x)_{e_{i_4}}-\sqrt{a_{i_1}a_{i_2}}\, (\mathbf y_x)_{e_{i_1}} (\mathbf y_x)_{e_{i_2}}. 
\end{aligned}\]
 Suppose, to the contrary, that $h(\mathbf y_x)<0$. If
\[\sqrt{a_{i_3}}\bigl|(\mathbf y_x)_{e_{i_3}}\bigr|=\max_{j\in\{1,2,3\}}\sqrt{a_{i_j}}\bigl|(\mathbf y_x)_{e_{i_j}}\bigr|,\]
set $\mathbf z_x:=|\mathbf y_x|$. Otherwise, let $t\in \{1,3\}$ such that \[\sqrt{a_{i_t}}\bigl|(\mathbf y_x)_{e_{i_t}}\bigr|=\max_{j\in\{1,2,3\}}\sqrt{a_{i_j}}\bigl|(\mathbf y_x)_{e_{i_j}}\bigr|,\]
and  define $\mathbf z_x\in\mathbb R^5$ by
\[(\mathbf z_x)_{e_{i_j}}=\begin{cases}-|(\mathbf y_x)_{e_{i_j}}|, & j\in \{3,4,5,t\},\\|(\mathbf y_x)_{e_{i_j}}|, & \text{otherwise}.
\end{cases}
\]
In either case,
\[h(\mathbf z_x)\ge0>h(\mathbf y_x),\]
while
\[\|\mathbf z_x\|=\|\mathbf y_x\|=1,\qquad\sum_{i=1}^{5} a_i (\mathbf z_x)_{e_i}^2=\sum_{i=1}^{5} a_i (\mathbf y_x)_{e_i}^2.\]
Consequently,
\[\begin{aligned}\mathbf z_x^{\top}M(x)\mathbf z_x
		&=	\sum_{i=1}^{5} a_i (\mathbf z_x)_{e_i}^2+\frac{2}{x+n-2}h(\mathbf z_x)\\
		&>\sum_{i=1}^{5} a_i (\mathbf y_x)_{e_i}^2+\frac{2}{x+n-2}h(\mathbf y_x)\\
		&=\mathbf y_x^{\top}M(x)\mathbf y_x\\
		&=g_1(x).\end{aligned}\]

\item {For the configurations (17).}

We fix an orientation satisfying {\rm (i)} and {\rm (ii)} such that the edges in $S$ are oriented as depicted in Table~\ref{tab-f}. Under this orientation, for any $x>0$,
\[
\begin{aligned}
	h(\mathbf y_x) &=\sqrt{a_{i_4}a_{i_5}}\, (\mathbf y_x)_{e_{i_4}} (\mathbf y_x)_{e_{i_5}}+\sqrt{a_{i_2}a_{i_4}}\, (\mathbf y_x)_{e_{i_2}} (\mathbf y_x)_{e_{i_4}}+  \sqrt{a_{i_2}a_{i_3}}\, (\mathbf y_x)_{e_{i_2}} (\mathbf y_x)_{e_{i_3}}	\\	&\quad + \sqrt{a_{i_1}a_{i_4}}\, (\mathbf y_x)_{e_{i_1}} (\mathbf y_x)_{e_{i_4}} -\sqrt{a_{i_1}a_{i_2}}\, (\mathbf y_x)_{e_{i_1}} (\mathbf y_x)_{e_{i_2}}. 
\end{aligned}\]
Suppose, to the contrary, that $h(\mathbf y_x)<0$. If
\[\sqrt{a_{i_4}}\bigl|(\mathbf y_x)_{e_{i_4}}\bigr|
\neq
\min_{j\in\{1,2,4\}}
\sqrt{a_{i_j}}\bigl|(\mathbf y_x)_{e_{i_j}}\bigr|,
\]
set $\mathbf z_x:=|\mathbf y_x|$. Otherwise, define $\mathbf z_x\in\mathbb R^5$ by
\[(\mathbf z_x)_{e_{i_j}}=
\begin{cases}
	-|(\mathbf y_x)_{e_{i_j}}|, & j=1\\
	|(\mathbf y_x)_{e_{i_j}}|, &\text{otherwise}.
\end{cases}
\]
In either case,
\[h(\mathbf z_x)\ge0>h(\mathbf y_x),\]
while
\[\|\mathbf z_x\|=\|\mathbf y_x\|=1,\qquad\sum_{i=1}^{5} a_i (\mathbf z_x)_{e_i}^2=\sum_{i=1}^{5} a_i (\mathbf y_x)_{e_i}^2.\]
Consequently,
\[\begin{aligned}\mathbf z_x^{\top}M(x)\mathbf z_x
		&=	\sum_{i=1}^{5} a_i (\mathbf z_x)_{e_i}^2+\frac{2}{x+n-2}h(\mathbf z_x)\\
		&>\sum_{i=1}^{5} a_i (\mathbf y_x)_{e_i}^2+\frac{2}{x+n-2}h(\mathbf y_x)\\
		&=\mathbf y_x^{\top}M(x)\mathbf y_x\\
		&=g_1(x).\end{aligned}\]

\item {For the configurations (18).}

We fix an orientation satisfying {\rm (i)} and {\rm (ii)} such that the edges in $S$ are oriented as depicted in Table~\ref{tab-f}. Under this orientation, for any $x>0$,
\[
\begin{aligned}
	h(\mathbf y_x) &=\sqrt{a_{i_2}a_{i_5}}\, (\mathbf y_x)_{e_{i_2}} (\mathbf y_x)_{e_{i_5}}+\sqrt{a_{i_2}a_{i_4}}\, (\mathbf y_x)_{e_{i_2}} (\mathbf y_x)_{e_{i_4}}+  \sqrt{a_{i_2}a_{i_3}}\, (\mathbf y_x)_{e_{i_2}} (\mathbf y_x)_{e_{i_3}}	\\	&\quad +\sqrt{a_{i_1}a_{i_2}}\, (\mathbf y_x)_{e_{i_1}} (\mathbf y_x)_{e_{i_2}} + \sqrt{a_{i_1}a_{i_4}}\, (\mathbf y_x)_{e_{i_1}} (\mathbf y_x)_{e_{i_4}} -\sqrt{a_{i_1}a_{i_5}}\, (\mathbf y_x)_{e_{i_1}} (\mathbf y_x)_{e_{i_5}}	\\	&\quad - \sqrt{a_{i_3}a_{i_4}}\, (\mathbf y_x)_{e_{i_3}} (\mathbf y_x)_{e_{i_4}} . 
\end{aligned}\]
We claim that $h(\mathbf y_x)\ge0$. Suppose, to the contrary, that $h(\mathbf y_x)<0$. If \[\sqrt{a_{i_2}}\bigl|(\mathbf y_x)_{e_{i_2}}\bigr|=\max_{j\in\{1,2,5\}}\sqrt{a_{i_j}}\bigl|(\mathbf y_x)_{e_{i_j}}\bigr|=\max_{j\in\{2,3,4\}}\sqrt{a_{i_j}}\bigl|(\mathbf y_x)_{e_{i_j}}\bigr|,\] then set $\mathbf z_x:=|\mathbf y_x|$. 
If \[\sqrt{a_{i_2}}\bigl|(\mathbf y_x)_{e_{i_2}}\bigr|\ne\max_{j\in\{1,2,5\}}\sqrt{a_{i_j}}\bigl|(\mathbf y_x)_{e_{i_j}}\bigr|\ne \max_{j\in\{2,3,4\}}\sqrt{a_{i_j}}\bigl|(\mathbf y_x)_{e_{i_j}}\bigr|,\] then let $t_1\in \{1,5\}$ and $t_2\in \{3,4\}$ such that \[\sqrt{a_{i_{t_1}}}\bigl|(\mathbf y_x)_{e_{i_{t_1}}}\bigr|=\max_{j\in\{1,2,5\}}\sqrt{a_{i_j}}\bigl|(\mathbf y_x)_{e_{i_j}}\bigr|\text{ and }\sqrt{a_{i_{t_2}}}\bigl|(\mathbf y_x)_{e_{i_{t_2}}}\bigr|=\max_{j\in\{2,3,4\}}\sqrt{a_{i_j}}\bigl|(\mathbf y_x)_{e_{i_j}}\bigr|.\]
Define $\mathbf z_x\in\mathbb R^5$ by
\[(\mathbf z_x)_{e_{i_j}}=
\begin{cases}
	-|(\mathbf y_x)_{e_{i_j}}|, & j\in \{1,3,4,5\}\setminus\{t_1,t_2\}\\
	|(\mathbf y_x)_{e_{i_j}}|, & \text{otherwise}.
\end{cases}
\]
If \[\sqrt{a_{i_2}}\bigl|(\mathbf y_x)_{e_{i_2}}\bigr|=\max_{j\in\{1,2,5\}}\sqrt{a_{i_j}}\bigl|(\mathbf y_x)_{e_{i_j}}\bigr|\ne \max_{j\in\{2,3,4\}}\sqrt{a_{i_j}}\bigl|(\mathbf y_x)_{e_{i_j}}\bigr|,\] then let $t_2\in \{3,4\}$ such that \[\sqrt{a_{i_{t_2}}}\bigl|(\mathbf y_x)_{e_{i_{t_2}}}\bigr|=\max_{j\in\{2,3,4\}}\sqrt{a_{i_j}}\bigl|(\mathbf y_x)_{e_{i_j}}\bigr|,\] and define $\mathbf z_x\in\mathbb R^5$ by
\[(\mathbf z_x)_{e_{i_j}}=
\begin{cases}
	-|(\mathbf y_x)_{e_{i_j}}|, & j\in \{3,4\}\setminus\{t_2\}\\
	|(\mathbf y_x)_{e_{i_j}}|, & \text{otherwise}.
\end{cases}
\]
If \[\sqrt{a_{i_2}}\bigl|(\mathbf y_x)_{e_{i_2}}\bigr|= \max_{j\in\{2,3,4\}}\sqrt{a_{i_j}}\bigl|(\mathbf y_x)_{e_{i_j}}\bigr|\ne \max_{j\in\{1,2,5\}}\sqrt{a_{i_j}}\bigl|(\mathbf y_x)_{e_{i_j}}\bigr|,\] then let $t_1\in \{1,5\}$ such that \[\sqrt{a_{i_{t_1}}}\bigl|(\mathbf y_x)_{e_{i_{t_1}}}\bigr|=\max_{j\in\{1,2,5\}}\sqrt{a_{i_j}}\bigl|(\mathbf y_x)_{e_{i_j}}\bigr|,\] and define $\mathbf z_x\in\mathbb R^5$ by
\[(\mathbf z_x)_{e_{i_j}}=
\begin{cases}
	-|(\mathbf y_x)_{e_{i_j}}|, & j\in \{1,5\}\setminus\{t_1\}\\
	|(\mathbf y_x)_{e_{i_j}}|, & \text{otherwise}.
\end{cases}
\]
In all cases,
\[h(\mathbf z_x)\ge0>h(\mathbf y_x),\]
while
\[\|\mathbf z_x\|=\|\mathbf y_x\|=1,\qquad\sum_{i=1}^{5} a_i (\mathbf z_x)_{e_i}^2=\sum_{i=1}^{5} a_i (\mathbf y_x)_{e_i}^2.\]
Consequently,
\[\begin{aligned}\mathbf z_x^{\top}M(x)\mathbf z_x
		&=	\sum_{i=1}^{5} a_i (\mathbf z_x)_{e_i}^2+\frac{2}{x+n-2}h(\mathbf z_x)\\
		&>\sum_{i=1}^{5} a_i (\mathbf y_x)_{e_i}^2+\frac{2}{x+n-2}h(\mathbf y_x)\\
		&=\mathbf y_x^{\top}M(x)\mathbf y_x\\
		&=g_1(x).\end{aligned}\]

\item {For the configurations (19).}

We fix an orientation satisfying {\rm (i)} and {\rm (ii)} such that the edges in $S$ are oriented as depicted in Table~\ref{tab-f}. Under this orientation, for any $x>0$,
\[
\begin{aligned}
	h(\mathbf y_x) &=  \sqrt{a_{i_2}a_{i_3}}\, (\mathbf y_x)_{e_{i_2}} (\mathbf y_x)_{e_{i_3}}+ \sqrt{a_{i_3}a_{i_4}}\, (\mathbf y_x)_{e_{i_3}} (\mathbf y_x)_{e_{i_4}}+ \sqrt{a_{i_1}a_{i_4}}\, (\mathbf y_x)_{e_{i_1}} (\mathbf y_x)_{e_{i_4}}	\\	&\quad + \sqrt{a_{i_1}a_{i_3}}\, (\mathbf y_x)_{e_{i_1}} (\mathbf y_x)_{e_{i_3}}+ \sqrt{a_{i_2}a_{i_5}}\, (\mathbf y_x)_{e_{i_2}} (\mathbf y_x)_{e_{i_5}}
	-\sqrt{a_{i_1}a_{i_2}}\, (\mathbf y_x)_{e_{i_1}} (\mathbf y_x)_{e_{i_2}}. 
\end{aligned}\]
Suppose, to the contrary, that $h(\mathbf y_x)<0$. If
\[
\sqrt{a_{i_3}}\bigl|(\mathbf y_x)_{e_{i_3}}\bigr|\ge \sqrt{a_{i_1}}\bigl|(\mathbf y_x)_{e_{i_1}}\bigr|
\]
set $\mathbf z_x:=|\mathbf y_x|$. Otherwise, define $\mathbf z_x\in\mathbb R^5$ by
\[
(\mathbf z_x)_{e_{i_j}}
=\begin{cases}
	-|(\mathbf y_x)_{e_{i_j}}|, & j\in\{1,3,4\}\\
	|(\mathbf y_x)_{e_{i_j}}|, & \text{ otherwise}.
\end{cases}
\]
In either case,
\[h(\mathbf z_x)\ge0>h(\mathbf y_x),\]
while
\[\|\mathbf z_x\|=\|\mathbf y_x\|=1,\qquad\sum_{i=1}^{5} a_i (\mathbf z_x)_{e_i}^2=\sum_{i=1}^{5} a_i (\mathbf y_x)_{e_i}^2.\]
\[\begin{aligned}\mathbf z_x^{\top}M(x)\mathbf z_x
		&=	\sum_{i=1}^{5} a_i (\mathbf z_x)_{e_i}^2+\frac{2}{x+n-2}h(\mathbf z_x)\\
		&>\sum_{i=1}^{5} a_i (\mathbf y_x)_{e_i}^2+\frac{2}{x+n-2}h(\mathbf y_x)\\
		&=\mathbf y_x^{\top}M(x)\mathbf y_x\\
		&=g_1(x).\end{aligned}\]

\item {For the configurations (20).}

We fix an orientation satisfying {\rm (i)} and {\rm (ii)} such that the edges in $S$ are oriented as depicted in Table~\ref{tab-f}. Under this orientation, for any $x>0$,
\[
\begin{aligned}
	h(\mathbf y_x) &=\sqrt{a_{i_1}a_{i_5}}\, (\mathbf y_x)_{e_{i_1}} (\mathbf y_x)_{e_{i_5}} ++\sqrt{a_{i_1}a_{i_4}}\, (\mathbf y_x)_{e_{i_1}} (\mathbf y_x)_{e_{i_4}} +	\sqrt{a_{i_4}a_{i_5}}\, (\mathbf y_x)_{e_{i_4}} (\mathbf y_x)_{e_{i_5}}\\	&\quad +	\sqrt{a_{i_2}a_{i_5}}\, (\mathbf y_x)_{e_{i_2}} (\mathbf y_x)_{e_{i_5}}+\sqrt{a_{i_3}a_{i_4}}\, (\mathbf y_x)_{e_{i_3}} (\mathbf y_x)_{e_{i_4}}-\sqrt{a_{i_1}a_{i_3}}\, (\mathbf y_x)_{e_{i_1}} (\mathbf y_x)_{e_{i_3}}	\\	&\quad -\sqrt{a_{i_1}a_{i_2}}\, (\mathbf y_x)_{e_{i_1}} (\mathbf y_x)_{e_{i_2}}. 
\end{aligned}\]
Suppose, to the contrary, that $h(\mathbf y_x)<0$. If
\[\sqrt{a_{i_4}}\bigl|(\mathbf y_x)_{e_{i_4}}\bigr|+\sqrt{a_{i_5}}\bigl|(\mathbf y_x)_{e_{i_5}}\bigr|\ge \sqrt{a_{i_2}}\bigl|(\mathbf y_x)_{e_{i_2}}\bigr|+\sqrt{a_{i_3}}\bigl|(\mathbf y_x)_{e_{i_3}}\bigr|,\]
then set $\mathbf z_x:=|\mathbf y_x|$. Otherwise, define $\mathbf z_x\in\mathbb R^5$ by
\[(\mathbf z_x)_{e_{i_j}}=\begin{cases}
	-|(\mathbf y_x)_{e_{i_j}}|, & j=1\\
	|(\mathbf y_x)_{e_{i_j}}|, & \text{otherwise}.
\end{cases}\]
In either case,
\[h(\mathbf z_x)\ge0>h(\mathbf y_x),\]
while
\[\|\mathbf z_x\|=\|\mathbf y_x\|=1,\qquad\sum_{i=1}^{5} a_i (\mathbf z_x)_{e_i}^2=\sum_{i=1}^{5} a_i (\mathbf y_x)_{e_i}^2.\]
Consequently,
\[\begin{aligned}\mathbf z_x^{\top}M(x)\mathbf z_x
		&=	\sum_{i=1}^{5} a_i (\mathbf z_x)_{e_i}^2+\frac{2}{x+n-2}h(\mathbf z_x)\\
		&>\sum_{i=1}^{5} a_i (\mathbf y_x)_{e_i}^2+\frac{2}{x+n-2}h(\mathbf y_x)\\
		&=\mathbf y_x^{\top}M(x)\mathbf y_x\\
		&=g_1(x).\end{aligned}\]

\item {For the configurations (21).}

We fix an orientation satisfying {\rm (i)} and {\rm (ii)} such that the edges in $S$ are oriented as depicted in Table~\ref{tab-f}. Under this orientation, for any $x>0$,
\[
\begin{aligned}
	h(\mathbf y_x) &=\sqrt{a_{i_1}a_{i_5}}\, (\mathbf y_x)_{e_{i_1}} (\mathbf y_x)_{e_{i_5}} +\sqrt{a_{i_1}a_{i_4}}\, (\mathbf y_x)_{e_{i_1}} (\mathbf y_x)_{e_{i_4}} +\sqrt{a_{i_4}a_{i_5}}\, (\mathbf y_x)_{e_{i_4}} (\mathbf y_x)_{e_{i_5}}\\	&\quad +	\sqrt{a_{i_2}a_{i_5}}\, (\mathbf y_x)_{e_{i_2}} (\mathbf y_x)_{e_{i_5}}+\sqrt{a_{i_3}a_{i_4}}\, (\mathbf y_x)_{e_{i_3}} (\mathbf y_x)_{e_{i_4}}+	\sqrt{a_{i_2}a_{i_3}}\, (\mathbf y_x)_{e_{i_2}} (\mathbf y_x)_{e_{i_3}}\\	&\quad-\sqrt{a_{i_1}a_{i_3}}\, (\mathbf y_x)_{e_{i_1}} (\mathbf y_x)_{e_{i_3}} -\sqrt{a_{i_1}a_{i_2}}\, (\mathbf y_x)_{e_{i_1}} (\mathbf y_x)_{e_{i_2}}. 
\end{aligned}\]
 Suppose, to the contrary, that $h(\mathbf y_x)<0$. If \[\sqrt{a_{i_4}}\bigl|(\mathbf y_x)_{e_{i_4}}\bigr|+\sqrt{a_{i_5}}\bigl|(\mathbf y_x)_{e_{i_5}}\bigr|\ge \sqrt{a_{i_2}}\bigl|(\mathbf y_x)_{e_{i_2}}\bigr|+\sqrt{a_{i_3}}\bigl|(\mathbf y_x)_{e_{i_3}}\bigr|,\] set $\mathbf z_x:=|\mathbf y_x|$. Otherwise, define $\mathbf z_x\in\mathbb R^5$ by
\[(\mathbf z_x)_{e_{i_j}}=\begin{cases}
	-|(\mathbf y_x)_{e_{i_j}}|, & j=1\\
	|(\mathbf y_x)_{e_{i_j}}|, & j=2,3,4,5.
\end{cases}
\]
In either case,
\[h(\mathbf z_x)\ge0>h(\mathbf y_x),\]
while
\[\|\mathbf z_x\|=\|\mathbf y_x\|=1,\qquad\sum_{i=1}^{5} a_i (\mathbf z_x)_{e_i}^2=\sum_{i=1}^{5} a_i (\mathbf y_x)_{e_i}^2.\]
Consequently,
\[\begin{aligned}\mathbf z_x^{\top}M(x)\mathbf z_x
		&=	\sum_{i=1}^{5} a_i (\mathbf z_x)_{e_i}^2+\frac{2}{x+n-2}h(\mathbf z_x)\\
		&>\sum_{i=1}^{5} a_i (\mathbf y_x)_{e_i}^2+\frac{2}{x+n-2}h(\mathbf y_x)\\
		&=\mathbf y_x^{\top}M(x)\mathbf y_x\\
		&=g_1(x).\end{aligned}\]

\item {For the configurations (22).}

We fix an orientation satisfying {\rm (i)} and {\rm (ii)} such that the edges in $S$ are oriented as shown in Table~\ref{tab-f}. Under this orientation,
\[
\begin{aligned}
	h(\mathbf y_x)
	&=
	\sqrt{a_{i_1}a_{i_4}}\,(\mathbf y_x)_{e_{i_1}}(\mathbf y_x)_{e_{i_4}}
	+\sqrt{a_{i_4}a_{i_5}}\,(\mathbf y_x)_{e_{i_4}}(\mathbf y_x)_{e_{i_5}}
	+\sqrt{a_{i_3}a_{i_5}}\,(\mathbf y_x)_{e_{i_3}}(\mathbf y_x)_{e_{i_5}}
	\\
	&\quad
	+\sqrt{a_{i_2}a_{i_5}}\,(\mathbf y_x)_{e_{i_2}}(\mathbf y_x)_{e_{i_5}}
	+\sqrt{a_{i_1}a_{i_5}}\,(\mathbf y_x)_{e_{i_1}}(\mathbf y_x)_{e_{i_5}}
	-\sqrt{a_{i_1}a_{i_3}}\,(\mathbf y_x)_{e_{i_1}}(\mathbf y_x)_{e_{i_3}}
	\\
	&\quad
	-\sqrt{a_{i_1}a_{i_2}}\,(\mathbf y_x)_{e_{i_1}}(\mathbf y_x)_{e_{i_2}}
	-\sqrt{a_{i_2}a_{i_3}}\,(\mathbf y_x)_{e_{i_2}}(\mathbf y_x)_{e_{i_3}} .
\end{aligned}
\]
Suppose, to the contrary, that $h(\mathbf y_x)<0$.
If \[\sqrt{a_{i_5}}\bigl|(\mathbf y_x)_{e_{i_5}}\bigr|=\max_{j\in\{1,2,3,5\}}\sqrt{a_{i_j}}\bigl|(\mathbf y_x)_{e_{i_j}}\bigr|,\]
set $\mathbf z_x:=|\mathbf y_x|$.
If \[\sqrt{a_{i_1}}\bigl|(\mathbf y_x)_{e_{i_1}}\bigr|=\max_{j\in\{1,2,3,5\}}\sqrt{a_{i_j}}\bigl|(\mathbf y_x)_{e_{i_j}}\bigr|,\]
define $\mathbf z_x\in\mathbb R^5$ by
\[(\mathbf z_x)_{e_{i_j}}=\begin{cases}-|(\mathbf y_x)_{e_{i_j}}|, & j\in\{1,4,5\},\\[1mm]	\phantom{-}|(\mathbf y_x)_{e_{i_j}}|, & \text{otherwise}.
\end{cases}\]
Finally, suppose that $ t\in\{2,3\}$ satisfies
\[\sqrt{a_{i_t}}\bigl|(\mathbf y_x)_{e_{i_t}}\bigr|=\max_{j\in\{1,2,3,5\}}\sqrt{a_{i_j}}\bigl|(\mathbf y_x)_{e_{i_j}}\bigr|.\]
If
\[
\sqrt{a_{i_1}}\bigl|(\mathbf y_x)_{e_{i_1}}\bigr|
\ge
\sqrt{a_{i_5}}\bigl|(\mathbf y_x)_{e_{i_5}}\bigr|,
\]
define
\[
(\mathbf z_x)_{e_{i_j}}
=
\begin{cases}
	-|(\mathbf y_x)_{e_{i_j}}|, & j=t,\\[1mm]
	\phantom{-}|(\mathbf y_x)_{e_{i_j}}|, & \text{otherwise},
\end{cases}
\]
whereas if
\[
\sqrt{a_{i_1}}\bigl|(\mathbf y_x)_{e_{i_1}}\bigr|
<
\sqrt{a_{i_5}}\bigl|(\mathbf y_x)_{e_{i_5}}\bigr|,
\]
define
\[
(\mathbf z_x)_{e_{i_j}}
=
\begin{cases}
	-|(\mathbf y_x)_{e_{i_j}}|, & j\in\{1,t,4,5\},\\[1mm]
	\phantom{-}|(\mathbf y_x)_{e_{i_j}}|, & \text{otherwise}.
\end{cases}
\]
A direct verification shows that, in each of the above cases,
\[
h(\mathbf z_x)\ge0>h(\mathbf y_x).
\]
Moreover, since $\|\mathbf z_x\|=\|\mathbf y_x\|=1$ and $\sum_{i=1}^{5} a_i (\mathbf z_x)_{e_i}^2=\sum_{i=1}^{5} a_i (\mathbf y_x)_{e_i}^2$,
we have
\[\begin{aligned}\mathbf z_x^{\top}M(x)\mathbf z_x
		&=	\sum_{i=1}^{5} a_i (\mathbf z_x)_{e_i}^2+\frac{2}{x+n-2}h(\mathbf z_x)\\
		&>\sum_{i=1}^{5} a_i (\mathbf y_x)_{e_i}^2+\frac{2}{x+n-2}h(\mathbf y_x)\\
		&=\mathbf y_x^{\top}M(x)\mathbf y_x\\
		&=g_1(x).\end{aligned}\]

\item {For the configurations (23).}

We fix an orientation satisfying {\rm (i)} and {\rm (ii)} such that the edges in $S$ are oriented as shown in Table~\ref{tab-f}. Under this orientation,
\[
\begin{aligned}
	h(\mathbf y_x)
	&=\sqrt{a_{i_4}a_{i_5}}\,(\mathbf y_x)_{e_{i_4}}(\mathbf y_x)_{e_{i_5}}+\sqrt{a_{i_3}a_{i_5}}\,(\mathbf y_x)_{e_{i_3}}(\mathbf y_x)_{e_{i_5}}+\sqrt{a_{i_2}a_{i_5}}\,(\mathbf y_x)_{e_{i_2}}(\mathbf y_x)_{e_{i_5}}\\
	&\quad+\sqrt{a_{i_1}a_{i_5}}\,(\mathbf y_x)_{e_{i_1}}(\mathbf y_x)_{e_{i_5}}-\sqrt{a_{i_3}a_{i_4}}\,(\mathbf y_x)_{e_{i_3}}(\mathbf y_x)_{e_{i_4}}-\sqrt{a_{i_2}a_{i_4}}\,(\mathbf y_x)_{e_{i_2}}(\mathbf y_x)_{e_{i_4}}\\
	&\quad-\sqrt{a_{i_1}a_{i_4}}\,(\mathbf y_x)_{e_{i_1}}(\mathbf y_x)_{e_{i_4}}-\sqrt{a_{i_1}a_{i_3}}\,(\mathbf y_x)_{e_{i_1}}(\mathbf y_x)_{e_{i_3}}
	-\sqrt{a_{i_1}a_{i_2}}\,(\mathbf y_x)_{e_{i_1}}(\mathbf y_x)_{e_{i_2}}
	\\
	&\quad-\sqrt{a_{i_2}a_{i_3}}\,(\mathbf y_x)_{e_{i_2}}(\mathbf y_x)_{e_{i_3}}.
\end{aligned}
\]
Suppose, to the contrary, that
$h(\mathbf y_x)<0$.
First consider the case \[\sqrt{a_{i_5}}\bigl|(\mathbf y_x)_{e_{i_5}}\bigr|=\max_{j\in\{1,2,3,4,5\}}\sqrt{a_{i_j}}\bigl|(\mathbf y_x)_{e_{i_j}}\bigr|.\] If \[\sqrt{a_{i_5}}\bigl|(\mathbf y_x)_{e_{i_5}}\bigr|\ge \sum_{j=2}^{4}\sqrt{a_{i_j}}\bigl|(\mathbf y_x)_{e_{i_j}}\bigr|,\]
then set $\mathbf z_x\mid =|\mathbf y_x|$.
Otherwise, 
define $\mathbf z_x\in\mathbb R^5$ by
\[(\mathbf z_x)_{e_{i_j}}=\begin{cases}	-|(\mathbf y_x)_{e_{i_j}}|, & j=1,\\[1mm]	|(\mathbf y_x)_{e_{i_j}}|, & \text{otherwise}.
\end{cases}
\]
Next consider the case \[\sqrt{a_{i_5}}\bigl|(\mathbf y_x)_{e_{i_5}}\bigr|\ne \max_{j\in\{1,2,3,4,5\}}\sqrt{a_{i_j}}\bigl|(\mathbf y_x)_{e_{i_j}}\bigr|,\] and choose $t\in\{1,2,3,4\}$ such that
\[\sqrt{a_{i_t}}\bigl|(\mathbf y_x)_{e_{i_t}}\bigr|=\max_{j\in\{1,2,3,4,5\}}\sqrt{a_{i_j}}\bigl|(\mathbf y_x)_{e_{i_j}}\bigr|.\]
If \[\sqrt{a_{i_t}}\bigl|(\mathbf y_x)_{e_{i_t}}\bigr|\ge \sum_{j\in\{1,2,3,4\}\setminus\{t\}}\sqrt{a_{i_j}}\bigl|(\mathbf y_x)_{e_{i_j}}\bigr|,\]
then define $\mathbf z_x\in\mathbb R^5$ by
\[(\mathbf z_x)_{e_{i_j}}=\begin{cases}	-|(\mathbf y_x)_{e_{i_j}}|, & j=\in \{t,5\},\\[1mm]	|(\mathbf y_x)_{e_{i_j}}|, & \text{otherwise}.
\end{cases}
\]
Otherwise, 
define $\mathbf z_x\in\mathbb R^5$ by
\[(\mathbf z_x)_{e_{i_j}}=\begin{cases}	-|(\mathbf y_x)_{e_{i_j}}|, & j=t,\\[1mm]	|(\mathbf y_x)_{e_{i_j}}|, & \text{otherwise}.
\end{cases}
\]
A direct verification shows that, in each of the above cases, $h(\mathbf z_x)\ge 0>h(\mathbf y_x)$.
Moreover, since $\|\mathbf z_x\|=\|\mathbf y_x\|=1$ and $\sum_{i=1}^{5} a_i (\mathbf z_x)_{e_i}^2=\sum_{i=1}^{5} a_i (\mathbf y_x)_{e_i}^2$, we have
\[\begin{aligned}\mathbf z_x^{\top}M(x)\mathbf z_x
		&=	\sum_{i=1}^{5} a_i (\mathbf z_x)_{e_i}^2+\frac{2}{x+n-2}h(\mathbf z_x)\\
		&>\sum_{i=1}^{5} a_i (\mathbf y_x)_{e_i}^2+\frac{2}{x+n-2}h(\mathbf y_x)\\
		&=\mathbf y_x^{\top}M(x)\mathbf y_x\\
		&=g_1(x).\end{aligned}\]

\item {For the configurations (24).}

We fix an orientation satisfying {\rm (i)} and {\rm (ii)} such that the edges in $S$ are oriented as depicted in Table~\ref{tab-f}. Under this orientation, for any $x>0$,
\[
\begin{aligned}
	h(\mathbf y_x) &=\sqrt{a_{i_4}a_{i_5}}\, (\mathbf y_x)_{e_{i_4}} (\mathbf y_x)_{e_{i_5}}+\sqrt{a_{i_1}a_{i_5}}\, (\mathbf y_x)_{e_{i_1}} (\mathbf y_x)_{e_{i_5}}+  \sqrt{a_{i_1}a_{i_3}}\, (\mathbf y_x)_{e_{i_1}} (\mathbf y_x)_{e_{i_3}}	\\	&\quad + \sqrt{a_{i_2}a_{i_5}}\, (\mathbf y_x)_{e_{i_2}} (\mathbf y_x)_{e_{i_5}}+\sqrt{a_{i_3}a_{i_4}}\, (\mathbf y_x)_{e_{i_3}} (\mathbf y_x)_{e_{i_4}} -\sqrt{a_{i_1}a_{i_2}}\, (\mathbf y_x)_{e_{i_1}} (\mathbf y_x)_{e_{i_2}}. 
\end{aligned}\]
Suppose, to the contrary, that $h(\mathbf y_x)<0$. If
\[
\sqrt{a_{i_5}}\bigl|(\mathbf y_x)_{e_{i_5}}\bigr|
\neq
\min_{j\in\{1,2,5\}}
\sqrt{a_{i_j}}\bigl|(\mathbf y_x)_{e_{i_j}}\bigr|,
\]
set $\mathbf z_x:=|\mathbf y_x|$. Otherwise, define $\mathbf z_x\in\mathbb R^5$ by
\[
(\mathbf z_x)_{e_{i_j}}
=
\begin{cases}
	|(\mathbf y_x)_{e_{i_j}}|, & j=1,3,4,5\\
	-|(\mathbf y_x)_{e_{i_j}}|, & j=2.
\end{cases}
\]
In either case,
\[h(\mathbf z_x)\ge0>h(\mathbf y_x),\]
while
\[\|\mathbf z_x\|=\|\mathbf y_x\|=1,\qquad\sum_{i=1}^{5} a_i (\mathbf z_x)_{e_i}^2=\sum_{i=1}^{5} a_i (\mathbf y_x)_{e_i}^2.\]
\[\begin{aligned}\mathbf z_x^{\top}M(x)\mathbf z_x
		&=	\sum_{i=1}^{5} a_i (\mathbf z_x)_{e_i}^2+\frac{2}{x+n-2}h(\mathbf z_x)\\
		&>\sum_{i=1}^{5} a_i (\mathbf y_x)_{e_i}^2+\frac{2}{x+n-2}h(\mathbf y_x)\\
		&=\mathbf y_x^{\top}M(x)\mathbf y_x\\
		&=g_1(x).\end{aligned}\]

\item {For the configurations (26).}

We fix an orientation satisfying {\rm (i)} and {\rm (ii)} such that the edges in $S$ are oriented as shown in Table~\ref{tab-f}. Under this orientation,
\[
\begin{aligned}
	h(\mathbf y_x)
	&=\sqrt{a_{i_4}a_{i_5}}\,(\mathbf y_x)_{e_{i_4}}(\mathbf y_x)_{e_{i_5}}
	+\sqrt{a_{i_3}a_{i_4}}\,(\mathbf y_x)_{e_{i_3}}(\mathbf y_x)_{e_{i_4}}
	+\sqrt{a_{i_1}a_{i_4}}\,(\mathbf y_x)_{e_{i_1}}(\mathbf y_x)_{e_{i_4}}\\
	&\quad
	+\sqrt{a_{i_2}a_{i_4}}\,(\mathbf y_x)_{e_{i_2}}(\mathbf y_x)_{e_{i_4}}
	-\sqrt{a_{i_1}a_{i_3}}\,(\mathbf y_x)_{e_{i_1}}(\mathbf y_x)_{e_{i_3}}\\
	&\quad
	-\sqrt{a_{i_1}a_{i_2}}\,(\mathbf y_x)_{e_{i_1}}(\mathbf y_x)_{e_{i_2}}
	-\sqrt{a_{i_2}a_{i_3}}\,(\mathbf y_x)_{e_{i_2}}(\mathbf y_x)_{e_{i_3}}.
\end{aligned}
\]
Suppose, to the contrary, that
$h(\mathbf y_x)<0$.
If \[\sqrt{a_{i_4}}\bigl|(\mathbf y_x)_{e_{i_4}}\bigr|=\max_{j\in\{1,2,3,4\}}\sqrt{a_{i_j}}\bigl|(\mathbf y_x)_{e_{i_j}}\bigr|,\]
then set $\mathbf z_x\mid =|\mathbf y_x|$.
Otherwise, choose $t\in\{1,2,3\}$ such that
\[\sqrt{a_{i_t}}\bigl|(\mathbf y_x)_{e_{i_t}}\bigr|=\max_{j\in\{1,2,3,4\}}\sqrt{a_{i_j}}\bigl|(\mathbf y_x)_{e_{i_j}}\bigr|,\] and define $\mathbf z_x\in\mathbb R^5$ by
\[(\mathbf z_x)_{e_{i_j}}=\begin{cases}	-|(\mathbf y_x)_{e_{i_j}}|, & j\in\{t,4,5\},\\[1mm]	|(\mathbf y_x)_{e_{i_j}}|, & \text{otherwise}.
\end{cases}
\]
A direct verification shows that, in each of the above cases, $h(\mathbf z_x)\ge 0>h(\mathbf y_x)$.
Moreover, since $\|\mathbf z_x\|=\|\mathbf y_x\|=1$ and $\sum_{i=1}^{5} a_i (\mathbf z_x)_{e_i}^2=\sum_{i=1}^{5} a_i (\mathbf y_x)_{e_i}^2$, we \[\begin{aligned}\mathbf z_x^{\top}M(x)\mathbf z_x
		&=	\sum_{i=1}^{5} a_i (\mathbf z_x)_{e_i}^2+\frac{2}{x+n-2}h(\mathbf z_x)\\
		&>\sum_{i=1}^{5} a_i (\mathbf y_x)_{e_i}^2+\frac{2}{x+n-2}h(\mathbf y_x)\\
		&=\mathbf y_x^{\top}M(x)\mathbf y_x\\
		&=g_1(x).\end{aligned}\]
\end{itemize}
\end{document}